\documentclass[12pt,twoside]{amsart}
\usepackage{graphicx,mathrsfs,epstopdf}
\usepackage{amsmath,amsfonts,amssymb,amsthm}
\usepackage[colorlinks,citecolor=blue,urlcolor=blue]{hyperref}
\usepackage[dvipsnames]{xcolor}
\usepackage{algorithmic}
\usepackage[ruled,noend]{algorithm2e}
\usepackage{framed} 
\usepackage{blkarray}
\usepackage{pgfplots}
\usepackage{natbib}
\usepackage{enumerate}
\usepackage{subfigure}
\usepackage{tikz}
\usetikzlibrary{chains,matrix,scopes,fit,decorations,positioning,shapes,snakes,arrows,decorations.markings}
\newtheorem{thm}{Theorem}[section]
\newtheorem{lem}[thm]{Lemma}
\newtheorem{prop}[thm]{Proposition}

\theoremstyle{definition}
\newtheorem{defn}[thm]{Definition}

\theoremstyle{remark}
\newtheorem{rem}[thm]{Remark}

\DeclareMathOperator{\rk}{rank}

\DeclareMathOperator{\cov}{cov} 
\DeclareMathOperator{\tw}{tw} 

\newcommand{\cC}{\mathcal{C}}

\newcommand{\E}{\mathbb{E}}
\newcommand{\cE}{\mathcal{E}}

\newcommand{\cS}{\mathcal{S}}
\newcommand{\N}{\mathbb{N}}

\newcommand{\cO}{\mathcal{O}}
\renewcommand{\P}{\mathbb{P}}

\newcommand{\R}{\mathbb{R}}
\renewcommand{\S}{\mathbb{S}}

\newcommand{\cT}{\mathcal{T}}

\newcommand{\cX}{\mathcal{X}}

\newcommand{\indic}{{1\!\!1}}

\newcommand{\bs}{\boldsymbol}

\newcommand{\cip}{\mbox{\,$\perp\!\!\!\perp$\,}}
\newcommand{\indep}{\cip}

\newcommand{\sn}{{\rm sn}}

\title[]{Property testing in graphical models: testing small separation numbers}

\author[]{Luc Devroye}
\address{}
\email{}
\thanks{}

\author[]{G\'{a}bor Lugosi}
\address{}
\email{}
\thanks{}
\author[]{Piotr Zwiernik}
\address{}
\email{}
\thanks{}

\date{\today}                                            

\begin{document}
\begin{abstract}
  In many statistical applications, the dimension is too large to
  handle for standard high-dimensional machine learning
  procedures. This is particularly true for graphical models, where
  the interpretation of a large graph is difficult and learning its
  structure is often computationally impossible either because the
  underlying graph is not sufficiently sparse or the number of
  vertices is too large. To address this issue, we develop a procedure
  to test a property of a graph underlying a graphical model that
  requires only a subquadratic number of correlation queries (i.e., we
  require that the algorithm only can access a tiny fraction of the
  covariance matrix). This provides a conceptually simple test to
  determine whether the underlying graph is a tree or, more generally,
  if it has a small separation number, a quantity closely related to
  the treewidth of the graph.
The proposed method is a divide-and-conquer algorithm that can be applied to quite general 
graphical models.
\end{abstract}

\maketitle


\section{Introduction}

Graphical models have gained widespread popularity in the field of
statistics as a powerful tool for modeling complex multivariate data
structures \cite{lau96}. The graphical representation allows for a
clear visualization of the underlying dependencies between the
components of a random vector $X=(X_1,\ldots,X_n)$, where each node in
the graph corresponds to a component and edges between nodes
represent direct statistical dependence between the corresponding random variables. This makes it easier to
reason about and interpret the relationships between variables in the
model.

Among different types of graphical models, Gaussian graphical models are particularly popular and useful due to their simplicity and mathematical tractability. Let $G=([n],E)$ be a graph with $n$ nodes, $[n]=\{1,\ldots,n\}$, and edge set $E$. Denote by $\S^n_+$ the set of symmetric positive definite $n\times n$ matrices. We consider the set
\begin{equation}\label{eq:defMG}
M(G)\;=\;\{\Sigma\in \S^n_+: (\Sigma^{-1})_{ij}=0 \mbox{ if }i\neq j\mbox{ and }ij\notin E\}.	
\end{equation}
In Gaussian graphical models it is assumed that the covariance matrix $\Sigma=\cov(X)$ lies in $M(G)$  for some graph $G$. If $X$ has Gaussian distribution then the relationship
$$
(\Sigma^{-1})_{ij}=0\quad\Longleftrightarrow\quad  X_i\indep X_j|X_{[n]\setminus \{i,j\}}
$$
makes it easy to interpret the graph. 

The sparsity and clear statistical interpretation offered by graphical models in high-dimensional scenarios \cite{wainwright2008graphical, lau96} have made them widely applied in areas such as genetics, neuroimaging, and finance, where the underlying graph structure can provide valuable insights into complex interactions and relationships between variables. 

In modern applications, the high dimensionality of datasets often makes it difficult to interpret the underlying graph structure. In such cases, researchers may need to focus on specific graph statistics that capture the most relevant aspects of the data. Depending on the application, different graph measures may be of interest, such as the diameter of the graph, measures of connectivity, high-degree vertices, community structures, and so on; see \cite{neykov2019combinatorial} for a discussion. This shift in focus from learning and testing the underlying graph structure to learning and testing more high-level combinatorial features of the graph motivates also our article.

Another motivation for our work is that, as pointed out in \cite{lugosi2021learning}, the computational cost of order $n^2$ becomes a major obstacle in a growing number of applications where the number of variables $n$ is large. In such cases, even storing or writing down the covariance matrix or its estimate becomes impractical, making standard approaches unfeasible. This issue arises in various fields, such as biology, where the problem of reconstructing gene regulatory networks from large-scale gene expression data requires relying on pairwise notions of dependence due to computational complexity \cite{hwang,Chan082099,zhang2011inferring}. Another example of large networks is building human brain functional connectivity networks using functional MRI data, where the data are usually aggregated to obtain a dataset with a moderate number of variables that can be processed with current algorithms \cite{huang2010learning}. 

The goal of \cite{lugosi2021learning} was to provide scalable learning
algorithms for graphical models that relied on sequential queries of
the covariance matrix $\Sigma$.
{
In that paper, the authors show how one can efficiently learn the
structure of a graph with small treewidth (equivalently, small
separation number). In this paper we do not assume that the graph has
a small separation number but rather the focus is on testing this property.
The present article builds on the ideas of  \cite{lugosi2021learning}.
}
We now briefly describe our setup.

\subsection*{Property testing:} 
{
Property testing is a well-established paradigm in theoretical computer science
(\cite{ron2001property}, \cite{goldreich2017introduction}) in which one is interested in testing
whether a large graph has certain properties, but with a limited access to information about the graph.
In the standard model, one is allowed to query the presence of a limited number of edges.
Our setup is similar in the sense that the goal is to test whether the underlying graph has certain
properties (e.g., whether it is a tree), but in graphical model testing, one cannot directly 
infer the presence of any edge. Instead, one is allowed limited access to the entries of the covariance 
matrix whose inverse encodes the graph of interest. This presents an additional challenge to construct
efficient testing procedures.
}
The hypothesis testing framework has been applied for graphical models
mostly in the context of structure learning; see, for example,
\cite{drton2007multiple,klaassen2023uniform}. Testing local
substructures was studied by \cite{verzelen2009tests}. In this paper
we are interested in testing structural properties of the graph
$G$. Our focus is on 
{the separation number as defined below. 
The separation number is closely related to treewidth, a fundamental and well-understood 
quantity in structural and algorithmic graph theory.
In particular, graphs with small 
 bounded treewidth 
}
 have long been
of interest in machine learning due to the low computational cost of
inference in such
models~\cite{chandrasekaran2012complexity,karger2001learning,kwisthout2010necessity}. Moreover,
current heuristics of treewidth estimation in real-world data have
indicated small treewidth in various cases of interest
\cite{abu2016metric,adcock2013tree,maniu2019experimental}.

Property testing in graphical models, also known as combinatorial inference, has been recently studied in several papers \cite{neykov2019combinatorial,neykov2019property,chung2017exact,sun2018sketching,tan2020inferring}. These papers have  relatively little focus on the computational aspects.  

Let $G=(V,E)$ be a graph. The focus of this paper is on testing two
properties of $G$. The first one is whether $G$ is a tree. The second
concerns the separation numbers; see, e.g.,
\cite{dvovrak2019treewidth}. A graph $G=(V,E)$ has separation number
$\sn(G)$ less than or equal to $k$, if for every subset $W\subseteq V$
with $|W|\geq k+2$, there is a partition of $W$ into three sets
$S,A,A'$, such that $A,A' \neq \emptyset$, $|S|\leq k$,
$\max\{|A|,|A'|\}\leq \tfrac23 |W|$ and $S$ separates $A$ from $A'$ in
the subgraph $G_W$. In other words, every subset of vertices can be
separated by a small subset of nodes in a balanced way; see
Section~\ref{sec:treewidth} for a formal discussion. It is easy to see
that $\sn(G)=1$ if and only if $G$ is a tree.

Given a covariance matrix $\Sigma\in M(G)$ for some graph $G$, in this paper we consider the problem of testing
\begin{equation}\label{eg:testtree}
H_0:\;\;G \mbox{ is a tree}\qquad\mbox{against}\qquad H_1:\;\;G \mbox{ has a cycle}.	
\end{equation}
and, more generally,  for $k\geq 1$,
\begin{equation}\label{eq:test}
H_0:\;\;\sn(G)\leq k  \qquad\mbox{against}\qquad H_1:\;\;\sn(G)>k.	
\end{equation}
A testing procedure can make queries of the entries of  $\Sigma$. In other words, $\Sigma\in M(G)$ is a covariance matrix and the problem is to test separation properties
of the graph of the underlying graphical model. The tester may make queries of entries of the covariance matrix.
The goal is to minimize the \emph{query complexity}, that is, the number of entries of $\Sigma$ the 
tester queries before making a decision.
 
We construct testing procedures that follow the general scheme of property testing 
(see, e.g., \cite{goldreich2017introduction}). At each step, a randomized sequential
algorithm decides whether to continue sampling or finish and reject
the null hypothesis. In our tree testing procedure, if the null hypothesis is true, the procedure
\emph{never} rejects it. If the null is not true, the test eventually
detects it. In the generalization of our procedure to $k>1$ this is still approximately true in the sense made precise in Theorem~\ref{th:main3} and Theorem~\ref{th:main2}.

The problem of testing whether $G$ is a tree/forest has been already considered in \cite{neykov2019combinatorial}. In problems that are not too big, a natural procedure is to use the Chow-Liu algorithm \cite{chow1968approximating}, which finds the maximum likelihood tree. This procedure however requires building the maximum cost spanning tree, based on all the entries of $\Sigma$, or more generally, all mutual informations between all pairs of variables. As we explain now, in our paper, we study the situation when this is not possible.

\subsection*{Computational budget:}  We focus on the situation, where $n$ is very large, in which case the covariance matrix $\Sigma$ can be hard to handle. Accessing the whole matrix $\Sigma$ at once is out of our computational bounds, and in particular, computing the inverse $\Sigma^{-1}$ cannot be done directly. Following \cite{lugosi2021learning} we study a query model, in which it is possible to query a single entry of $\Sigma$ at a time. Such situations may either appear when various parts of $\Sigma$ are stored in different locations or when it is possible to design an experiment that allows one to query a particular part of the system. Needless to say, our approach can be also used in the case when $\Sigma$ can be stored and accessed easily but applying transformations to the whole matrix is costly. In the whole analysis we make sure that our procedures are within the desired computational budget. Our goal is to construct testing procedures  whose query complexity is $o(n^2)$.

The query model is natural for many property testing tasks in
graphical models. It has also been applied to the task of recovering
of other properties of $\Sigma$ from a few entries;
\cite{lugosi2021learning,khetan2019spectrum}. In particular, our
procedure of testing whether the underlying graph is a tree has
lower query complexity and is conceptually simpler than
the one in \cite{neykov2019combinatorial}. Moreover, it can be applied
to any distributional setting, as long as we can efficiently test
conditional independence queries $X_i\indep X_j|X_k$.


\subsection*{Main contributions} Here we summarize the main contributions of the paper.

Suppose $\Sigma\in M(G)$ is a covariance matrix in a Gaussian graphical model over the graph $G$. First we propose a randomized Algorithm~\ref{algo3} for testing if $H$ is a tree. 

\begin{thm}[Simplified version of Theorem~\ref{th:main1}]
Let $G$ be a connected graph
 and let $\Sigma$ be a (generic) covariance matrix in a Gaussian graphical model over $G$. 
{
Then, Algorithm~\ref{algo3} correctly identifies whether $G$ is a tree or not.
Moreover, with high probability, it runs with total query complexity	
$$
\cO(n\log(n)(\min\{\log^2n,\Delta\}+\Delta))~,
$$
where
$\Delta$ is the maximum degree of $G$.
}
\end{thm}
Although this result is formulated for Gaussian graphical models, we
show in Section~\ref{sec:ngtrees} that it can be applied in a number of other situations (e.g., for binary data). Moreover, our algorithm can be further extended to work for
\emph{any} distributional setting and it uses a minimal number of
conditional independence queries $X_i\indep X_j|X_k$, which in the
Gaussian case corresponds to checking the equality
$\Sigma_{ij}\Sigma_{kk}=\Sigma_{ij}\Sigma_{ik}$ involving only four entries of the covariance matrix--and therefore requiring only four queries.

Second, we study the problem of testing if the underlying graph has a small separation number. We discuss two algorithms. The first one  generalizes our procedure for trees. It has strong theoretical guarantees but it may become inconclusive for certain type of graphs. If the algorithm is conclusive, we say that it has a good run; see Definition~\ref{def:goodrun} for a formal discussion.

\begin{thm}[Simplified version of Theorem~\ref{th:main3}]
Let $G$ be a connected graph 
and let $\Sigma$ be a (generic) covariance matrix in a Gaussian
graphical model over $G$. Let $k$ be a positive integer and run Algorithm~\ref{algomainmarg}
with parameter $k$. Suppose that the algorithm has good run. If  the algorithm terminated, then $\sn(G)\leq 2k$ and if it broke then $\sn(G)>\tfrac23 k$.	
Moreover, with high probability, it
runs with total query complexity
$$
\cO(n\log(n)\max\{\log(n),k\Delta\})~.
$$
where
$\Delta$ is the maximum degree of $G$.
\end{thm}

Moreover, as we show in Theorem~\ref{th:decomp}, in the special case when $G$ is decomposable (or chordal), our algorithm always has a good run and it gives a definite answer whether $\sn(G)\leq k$ or not directly generalizing the tree case. A graph is decomposable if it has no induced cycles of size greater of equal to four; see, for example, Section~2.1.2 in \cite{lau96}. 

Our second algorithm works in all situations but has weaker theoretical guarantees.

\begin{thm}[Simplified version of Theorem~\ref{th:main2}]
Let $G$ be a connected graph 
and let $\Sigma$ be a (generic) covariance matrix in a Gaussian
graphical model over $G$. Let $k$ be a positive integer and run Algorithm~\ref{algomaincond}
with parameter $k$. If  the algorithm breaks then $\sn(G)> k$. If it terminates, then $\sn(G)\leq 10k\log(\tfrac{n}{k})$. Moreover, with high probability, it
runs with total query complexity
$$
\cO(n\log(n)\max\{\log(n),k\Delta\})~.
$$
where
$\Delta$ is the maximum degree of $G$.
\end{thm}

\subsection*{Notation}
We briefly introduce some notation used throughout the paper. Denote
by $\S^n$ the set of real symmetric $n\times n$ matrices and by
$\S^n_+$ its subset given by positive definite matrices. For
$\Sigma\in \S^n$, let $\Sigma_{A,B}$ denote the submatrix of $\Sigma$
with rows in $A\subseteq \{1,\ldots,n\}$ and columns in
$B\subseteq \{1,\ldots,n\}$. Writing $\Sigma_{\setminus i,i}$, we mean
taking $A=\{1,\ldots,n\}\setminus \{i\}$ and $B=\{i\}$. Similarly, for
a vector $x\in \R^n$ and $A\subseteq \{1,\ldots,n\}$ we write $x_A$ to
denote the subvector of $x$ with entries $x_i$ for $i\in A$.

For a graph $G=(V,E)$ and a subset $B\subseteq V$ denote by $G_B$
 the induced subgraph, that is, the subgraph of $G$ with vertex set $B$ and edge set given by all edges in
$E$ with both endpoints in $B$. To write that $G$ contains an edge
between vertices $i,j$ we write $ij\in G$.

  The rest of the paper is structured as follows.
In Section \ref{sec:faitful} we discuss the ``faithfulness''
assumptions that the covariance matrix is required to fulfil in order
for our procedures to work correctly. In particular, we point out that
the set of covariance matrices satisfying the requirements is an open
dense subset of the set of symmetric positive definite matrices. We
also introduce a simple procedure to determine the connected
components of the underlying graph.
In Section~\ref{sec:tree} we discuss the simplest special case, that
is, the problem of testing whether the underlying graph is a tree.
We state and prove the first main theorem (Theorem \ref{th:main1})
that shows that one can efficiently test whether $G$ is a
bounded-degree tree. In Section \ref{sec:treewidth} we present the
general algorithm for testing small separation number, culminating
in Theorems \ref{th:main3} and \ref{th:main2}, the main results of the paper.
Finally, in Section \ref{sec:ngtrees} we discuss extensions to more
general (not necessarily Gaussian) graphical models.

%

\section{Faithfulness and connected components}\label{sec:faitful}

\subsection{Separation in graphs}\label{sec:separation}

For a graph $G=(V,E)$ we say that $A,A'\subset V$ are separated by a
vertex set $S$ if every path between a vertex in $A$ and a vertex in
$A'$ contains a vertex in $S$. In other words, $A$ and $A'$ are
disconnected in the graph $G\setminus S$ obtained from $G$ by removing
the vertices in $S$ and all the incident edges.

In graphical models, vertices of the graph $G$ represent random variables and no edge between $i$ and $j$ implies conditional independence $X_i\indep X_j|X_{V\setminus \{i,j\}}$. For strictly positive densities, the Hammersley-Clifford theorem also implies that $X_i\indep X_j|X_S$ whenever $S$ separates $i$ from $j$ in $G$; see \cite{lau96} for more details.

In the Gaussian case the conditional independence $X_i\indep X_j|X_S$ is equivalent to $\Sigma^{(S)}_{ij}=0$,
where $\Sigma^{(S)}_{ij}$ is the $(i,j)$-th entry of the conditional covariance matrix
\begin{equation}\label{eq:SigmaS}
\Sigma^{(S)} \;:=\;\Sigma_{\overline S,\overline S}-\Sigma_{\overline S,S}\Sigma_{S,S}^{-1}\Sigma_{S,\overline S}\qquad \mbox{with }\overline S=V\setminus S.
\end{equation}
Equivalently, by the Guttman rank additivity formula (see (14) in \cite{lugosi2021learning}), $\rk(\Sigma_{A\cup S,A'\cup S})=|S|$, where $A,A'$ are any two subsets of nodes separated by $S$ in $G$. Note also that if $S=\{v\}$ for  $v\in [n]$, \eqref{eq:SigmaS} specializes to 
\begin{equation}\label{eq:Sigmav}
\Sigma^{(v)} \;:=\;\Sigma_{V\setminus \{v\}, V\setminus \{v\}}-\frac{1}{\Sigma_{v,v}}\Sigma_{V\setminus \{v\},\{v\}}\Sigma_{\{v\}, V\setminus \{v\}}
\end{equation}


All our procedures rely on a divide-and-conquer approach, where in each step the graph is divided into balanced components by a small separator $S$, in which the matrix $\Sigma^{(S)}$ plays a crucial role.

\subsection{Faithfulness and learning connected components}

It follows from the definition of $M(G)$ in \eqref{eq:defMG} that, if $H$ is a subgraph of $G$, then $M(H)\subset M(G)$. In particular, if $\Sigma$ is diagonal, then $\Sigma\in M(G)$ for every $G$. In order to be able to read from $\Sigma\in M(G)$ structural information about $G$, we need to require that $\Sigma$ is in some way generic in $M(G)$. In this paper we consider two such genericity conditions.
\begin{defn}\label{def:faith0}
	We say that $\Sigma\in M(G)$ is \emph{faithful} over $G$ if for any $A,A',S\subseteq V$ we have ${\rm rank}(\Sigma_{A\cup S,A'\cup S})=|S|$ if and only if $S$ separates $A,A'$ in $G$. 
\end{defn}
\noindent Note that the left implication in
Definition~\ref{def:faith0} always holds by the discussion in
Section~\ref{sec:separation} and faithfulness implies that the
opposite implication also holds. This is a genericity condition in the
sense that the set of $\Sigma\in M(G)$ do not satisfy this condition
is a union of explicit proper algebraic subsets of $M(G)$ and thus the
faithful $\Sigma$ form a dense subset of $M(G)$. 
\begin{defn}\label{def:stfaith0}
	We say that $\Sigma\in M(G)$ is \emph{strongly faithful} over $G$ if for any $A,A'\subseteq V$,  ${\rm rank}(\Sigma_{A,A'})$ equals the size of a minimal separator of $A,A'$ in $G$. Denote the set of strongly faithful covariance matrices by $M^\circ(G)$. 
\end{defn}
\noindent By Theorem~2.15 in \cite{sullivant2010trek},
$\rk(\Sigma_{A,A'})$ is always upper bounded by the size of a minimal separator and equality holds generically and so $M^\circ(G)$ forms a dense subset of $M(G)$. Moreover, strong faithfulness implies faithfulness. Indeed, if $S$ separates $A,A'$ then $S$ is the minimal separator of $A\cup S$ and $A'\cup S$. By strong faithfulness, $\rk(\Sigma_{A\cup S,A'\cup S})=|S|$. 

In applications, often one does not have access to a perfect oracle that returns 
exact values of the queried entries of the covariance matrix. In such cases,
 working with a ``noisy'' oracle, one needs to assume that if
$i$ and $j$ are not separated by $S$ then
$|\Sigma^{(S)}_{ij}|\geq \epsilon$ for some $\epsilon>0$. Assuming
such inequalities for too many instances may result in a small set of eligible covariance matrices
\cite{uhler2013geometry}. In some situations, our procedures will work under strictly weaker genericity conditions.

\begin{defn}\label{def:faith}
Fix $\tau\in \N$. The matrix $\Sigma\in M(G)$ is \emph{$\tau$-faithful to $G$} if the condition in Definition~\ref{def:faith0} holds whenever $|S|\leq \tau$. We write $\Sigma\in M^\tau(G)$. Moreover, $\Sigma$ is $\tau$-strongly faithful to $G$ if the condition in Definition~\ref{def:stfaith0} holds if $\rk(\Sigma_{A,A'})\leq \tau$. We write $\Sigma\in M^{\tau,\circ}(G)$.
\end{defn}  
\noindent For example, the matrix $\Sigma\in M(G)$ is \emph{$0$-faithful to $G$} if
  $\Sigma_{ij}=0$ is equivalent to $i$ and $j$ lying in two different
  connected components of $G$. The fact that $\tau$-strong faithfulness implies $\tau$-faithfulness follows by the same argument as above.



The role of the faithfulness assumptions is that
$\Sigma$ then encodes accurately the underlying graph or at least the
part we care about. The following fact gives the first instance of
how it can be exploited.
\begin{lem}\label{lem:M0coms}
If $\Sigma\in M^0(G)$ then $\Sigma$ has a block diagonal structure with blocks corresponding to the connected components of $G$ and each block has \emph{all non-zero} elements.	More generally, if $\Sigma\in M^\tau(G)$ and $|S|\leq \tau$ then $\Sigma^{(S)}\in M^0(G\setminus S)$.
\end{lem}

By Lemma~\ref{lem:M0coms}, if $\Sigma\in M^0(G)$, the connected component containing a vertex $i$ is equal to the support of the vector $\Sigma_{\setminus i,i}$ consisting of all the off-diagonal entries of the $i$-th column of $\Sigma$. This observation is a building block behind many useful procedures that we utilize later. For example, to identify all the connected components of $G$, we run the procedure \texttt{components}$(\Sigma,[n])$ as outlined in Algorithm~\ref{algo1}. The algorithm takes a random vertex $i$ and checks the support of  $\Sigma_{\setminus i,i}$. If this support is $C_1$, sample then a random vertex $j$ from $[n]\setminus C_1$ and check the support of $\Sigma_{j,[n]\setminus (j\cup C_1)}$ to get the component $C_2$ of $j$. This proceeds until the union of all the components obtained in this way is precisely $[n]$.
\begin{lem}\label{lem:comp}
	Let $G=(V,E)$ be a graph with connected components $C_1,\ldots,C_\ell$ and fix a subset of nodes $W\subseteq V$. If $\Sigma\in M^0(G)$, then the procedure \emph{\texttt{components}}$(\Sigma,W)$ finds the decomposition of $W$ into the sets $C_1\cap W,\ldots,C_\ell\cap W$. The algorithm uses $\mathcal O(|W|\min\{\ell,|W|\})$ covariance queries.\end{lem}
\begin{proof}
	The first part of the proof is clear in light of Lemma~\ref{lem:M0coms}. The algorithm has query complexity 
\begin{equation}\label{eq:complAl1}
	|C_1\cap W|+2|C_2\cap W|+\cdots +\ell|C_\ell\cap W|\;\;=\;\; \mathcal O(|W|\min\{\ell,|W|\}).
\end{equation} 
\end{proof}
%
\begin{algorithm}
Input: a subset $W\subseteq  [n]$, and oracle for $\Sigma$\;
Set $C_0:=\emptyset$, $\ell=0$, $A_0=W$\;
\While{$A_\ell\neq\emptyset$}{
Take a random vertex $i$ from $A_{\ell}$\; 
Set $\bs v=\Sigma_{A_{\ell}\setminus i,i}$\;
$C_{\ell+1}:={\tt supp}(\bs v)$\;
$A_{\ell+1}=A_{\ell}\setminus C_{\ell+1}$\;
$\ell := \ell + 1$\;
}
Return $C_1,\ldots,C_{\ell}$\;
\label{algo1}
\caption{${\tt Components}(\Sigma,W)$}
\end{algorithm}
Note that the expected complexity of the algorithm is typically smaller than that given in Lemma~\ref{lem:comp} because big components are likely to be detected earlier and so the sum on the left of \eqref{eq:complAl1} is smaller in expectation than $|W|\min\{\ell,|W|\}$. 

\begin{rem}
Directly by definition, Algorithm~\ref{algo1} can be also used to identify the components of $G\setminus S$ as long as $\Sigma^{(S)}$ is $0$-faithful (which holds as long as $\Sigma$ is $|S|$-faithful). In the tree case in Section~\ref{sec:tree} we only require that $\Sigma$ is 1-faithful. 	
\end{rem}

\textbf{Connectedness assumption:} From now on we assume that $G$ is connected. Equivalently, we restrict ourselves to a particular connected component. If the number of connected components is $o(n)$, we can identify them within our computational budget. In addition, the 
separation number of $G$ is the maximum of separation numbers of its connected components. 


\section{Testing if $G$ is a tree}\label{sec:tree}

Recall that our general goal is to test properties of the underlying graph $G$ with $n$ nodes, where the graph is encoded in the support of $\Sigma^{-1}$. In this section we present a simple randomized procedure, which tests whether $G$ is a tree.

Our testing procedure follows a divide-and-conquer approach. First an
approximately central cut-vertex $v^*$ is found (see
Section~\ref{sec:central}) and the procedure descends to the connected
components of $G\setminus \{v^*\}$ (see
Section~\ref{sec:descending}). In Figure~\ref{fig:levels} we provide a
simple example, where $G\setminus \{v^*\}$ has three components
$C_1,C_2,C_3$. Here centrality assures that each of these components
has size at most $2n/3$ with high probability. Now the second step of
the procedure starts, where the same step is repeated in each of the subsets
$B_i=C_i\cup \{v^*\}$.
If a component
$C$ is small enough (green nodes in Figure~\ref{fig:levels}) we query
the whole submatrix $\Sigma_{C,C}$ and run a direct check if this
component is tree supported. We do not descend in this component any
further. It may happen that we get evidence that one of the components
is not a tree (the red node in Figure~\ref{fig:levels}), which may
happen either because there is no cut-vertex in the given component or
the direct check in a small component fails. In this case the
algorithm stops and rejects the null hypothesis.

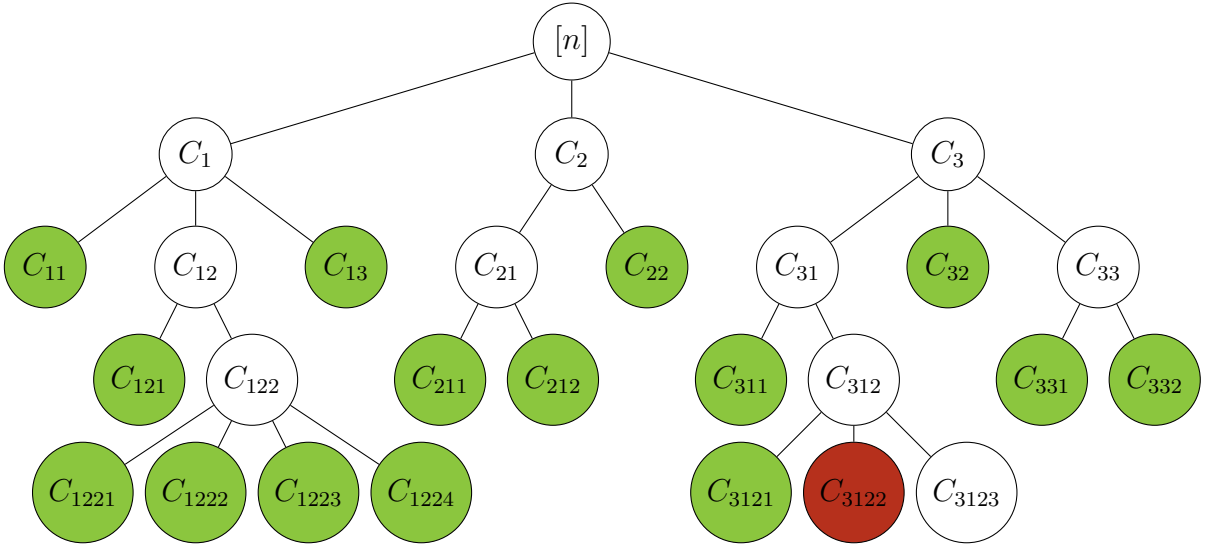
\begin{figure}
\centering
\begin{tikzpicture}[level 1/.style={sibling distance=50mm},
                    level 2/.style={sibling distance=20mm},
                    level 3/.style={sibling distance=15mm},
                    level 4/.style={sibling distance=15mm}]
\node (root) [circle, draw] {$[n]$}
  child {node [circle, draw] {$C_1$}
    child {node [circle, draw,fill=LimeGreen] {$C_{11}$}}
    child {node [circle, draw] {$C_{12}$}
      child {node [circle, draw,fill=LimeGreen] {$C_{121}$}}
      child {node [circle, draw] {$C_{122}$}
      	child {node [circle, draw,fill=LimeGreen] {$C_{1221}$}}
      	child {node [circle, draw,fill=LimeGreen] {$C_{1222}$}}
      	child {node [circle, draw,fill=LimeGreen] {$C_{1223}$}}
      	child {node [circle, draw,fill=LimeGreen] {$C_{1224}$}}
      	}
    }
    child {node [circle, draw,fill=LimeGreen] {$C_{13}$}
    }
  }
  child {node [circle, draw] {$C_2$}
    child {node [circle, draw] {$C_{21}$}
      child {node [circle, draw,fill=LimeGreen] {$C_{211}$}}
      child {node [circle, draw,fill=LimeGreen] {$C_{212}$}}
    }
    child {node [circle, draw,fill=LimeGreen] {$C_{22}$}
    }
  }
  child {node [circle, draw] {$C_3$}
    child {node [circle, draw] {$C_{31}$}
      child {node [circle, draw,fill=LimeGreen] {$C_{311}$}}
      child {node [circle, draw] {$C_{312}$}
      	child {node [circle, draw,fill=LimeGreen] {$C_{3121}$}}
      	child {node [circle, draw,fill=BrickRed] {$C_{3122}$}}
      	child {node [circle, draw] {$C_{3123}$}}}
    }
    child {node [circle, draw,fill=LimeGreen] {$C_{32}$}
    }
    child {node [circle, draw] {$C_{33}$}
      child {node [circle, draw,fill=LimeGreen] {$C_{331}$}}
      child {node [circle, draw,fill=LimeGreen] {$C_{332}$}}
    }
  };
\end{tikzpicture}
\caption{Schematic picture of our divide-and-conquer algorithm.}\label{fig:levels}
\end{figure}

\subsection{Finding a central vertex in $G$}\label{sec:central}
Our procedure starts by finding an approximately central vertex in the graph $G$.  We use a standard definition of the node centrality in trees \cite{jordan1869assemblages}. For $v\in V$, denote by $\cC^{(v)}$ the set of connected components of $G\setminus \{v\}$. Define
\begin{equation}\label{eq:Mv}
M(v)\;:=\;\max_{C\in \cC^{(v)}} |C|.	
\end{equation}
Then $v$ is central if it minimizes $M(v)$ over all $v\in V$. If the
minimum is attained more than once, pick one of the optimal vertices
arbitrarily. It is well known that, if $G$ is a tree, then
$\min_{v\in V} M(v)\leq n/2$ and the minimum is attained at most
twice (see \cite{harary6graph}).

Since $n$ is large, computing $M(v)$ directly exceeds our computational bounds. We need to find a reliable and computationally efficient method to estimate the maximal component size in each $G\setminus \{v\}$. This can be done as follows:
\begin{enumerate}
	\item [(S1)] Sample $m$ nodes uniformly at random $W=\{v_1,\ldots,v_m\}$ without replacement from $V$.
	\item [(S2)] For each $v\in V$ find the restricted decomposition $$\cC^{(v)}(W)\;:=\;\{C\cap W: C\in \cC^{(v)}\mbox{ and }C\cap W\neq \emptyset\}.$$
	This can be done by running {\tt Components}$(\Sigma^{(v)},W\setminus\{v\})$.
	\item [(S3)] Use the size of the largest element in $\cC^{(v)}(W)$ as the estimator of $M(v)$ in \eqref{eq:Mv} by defining \begin{equation}\label{eq:mhat}
		\widehat M(v)=n\max_{C\in \cC^{(v)}}\frac{|C\cap W|}{|W|},\qquad\mbox{and take}\qquad  v^*:=\arg\min_{v\in V} \widehat M(v).
	\end{equation}
\end{enumerate}
 The parameter $m$ is a computational budget parameter, which is required to be not too small (see \eqref{eq:mtree}). The whole procedure is outlined in Algorithm~\ref{algo2} and we now explain it in more detail.

\begin{algorithm}
Input: an oracle for $\Sigma$, $m\in \N$\;
Output: $v^*$ and $C_1,\ldots, C_k$\;
Let $V$ be the index set of the rows/columns of $\Sigma$\;
Sample $m$ nodes at random $W\subseteq V$ without replacement\;
\For{$v\in V$}{
run \texttt{Components}($\Sigma^{(v)},W\setminus \{v\}$) given in Algorithm~\ref{algo1} ($\Sigma^{(v)}$ defined in \eqref{eq:SigmaS})\;
compute $\widehat M(v)$ as in \eqref{eq:mhat}\;
}
\If{$\min_{v\in V} \widehat M(v)>\tfrac{|V|}{2}$}{{\tt BREAK} ($G$ is not a tree)\;}
\Else{Return $v^*=\arg\min_{v\in V} \widehat M(v)$ and components \texttt{Components}($\Sigma^{(v^*)},V\setminus \{v^*\}$)\;}
\label{algo2}
\caption{${\tt FindBalancedPartitionTree}(\Sigma,m)$}
\end{algorithm}

\begin{lem}\label{lem:compl2}Let $V$ be the set of vertices in the current call of the algorithm.	The query complexity of Algorithm~\ref{algo2} is 
	$$\cO(|V|\left(m\min\{m,\Delta\}+\Delta\right)\}),$$
	where  $\Delta$ is the maximal degree of $G$.
\end{lem}
\begin{proof}
The algorithm queries $\Sigma$ while running {\tt
  Components}$(\Sigma^{(v)},W\setminus \{v\})$. By
Lemma~\ref{lem:comp}, every call makes $\cO(m \ell)$ queries, where $\ell\leq m$ is the number of elements in $\cC^{(v)}(W)$. Since this number is bounded by $\min\{m,\Delta\}$, we need to query
$\mathcal O(m\min\{m,\Delta\})$ entries of $\Sigma^{(v)}$. Note that
{\color{purple}
by \eqref{eq:Sigmav},
}
to query a single element of $\Sigma^{(v)}$ we query at most four elements of $\Sigma$. Thus, the overall query 
complexity of this step is $\mathcal O(|V|m\min\{m,\Delta\})$, where the extra $|V|$ takes into account that we need to repeat the same operation for
each $v\in V$.

Once the minimizer $v^*$ is found, the partition $\cC^{(v)}(W)$ in step (S2) can be found efficiently
using Algorithm~\ref{algo1}. This has query complexity
$\mathcal O(\ell |V|)=\mathcal O(\Delta |V|)$, where $\ell$ is the number of elements in
$\cC^{(v)}(W)$. Thus, the overall query complexity for the whole algorithm is $\cO(|V|\left(m\min\{m,\Delta\}+\Delta\right)\})$.
\end{proof}

We still need to justify that the optimal separator $v^*$ for $W$, as defined in \eqref{eq:mhat}, remains a good separator of the whole $G$. Recall that, if $G$ is a
tree, then $\min_{v\in V} M(v)\leq n/2$. What is less clear is that
the same holds for $\widehat M(v)$, which explains the \texttt{BREAK}
line in Algorithm~\ref{algo2}. This fact relies on the following basic
result.

\begin{lem}\label{lem:partA}
Consider a tree $T=(V,E)$ with each node $v\in V$ having a weight $w(v)\geq 0$. Let $w(C):=\sum_{v\in C} w(v)$ for any $C\subseteq V$. Then there exists a \emph{central} node $v^*$ such that its removal splits the vertices into disjoint subsets $C_1,\ldots,C_\ell$ with 
\begin{equation}\label{eq:Ccond}
\max_i(w(C_i))\;\leq\;\frac{w(V)}{2}.	
\end{equation}
In particular, taking $w(v)=1$ if $v\in W$ and $w(v)=0$ otherwise, we get that that $\min_{v\in V}\widehat M(v)\leq \tfrac{n}{2}$.\end{lem}
%
%
%
\begin{figure}
	\includegraphics[scale=.6]{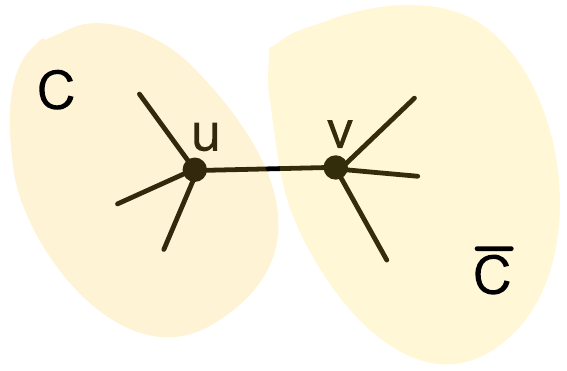}
	\caption{Illustration of the proof of Proposition~\ref{lem:partA}.}\label{fig:proofW}
\end{figure}
\begin{proof}
	Take any vertex $v$. If $v$ satisfies the condition \eqref{eq:Ccond}, we are done. So suppose $w(C)> \frac{w(V)}{2}$ for some component $C$ of $T\setminus\{v\}$. Let $u$ be the neighbor of $v$ that lies in $C$, and let $\overline C$ be the complement of $C$. Now consider the decomposition induced by $u$ and note that $\overline C$ becomes one of the corresponding connected components of $T\setminus \{u\}$; see Figure~\ref{fig:proofW}. We have  $w(\overline C)<\frac{w(V)}{2}$ for otherwise $w(V)=w(C)+w(\overline C)>w(V)$. If there is some component $C'$ of $u$ such that $w(C')>\tfrac{w(V)}{2}$ then  $w(C')\leq w(C)-w(u)\leq w(C)$. We can now  apply the same argument as above to $u$ making sure that, if $u$ does not satisfy \eqref{eq:Ccond}, we move to a uniquely defined neighbor but never returning to a previously visited vertex. In every such move, the size of the maximal component cannot increase and eventually it must decrease.
\end{proof}

Proposition~\ref{lem:partA} implies that, if
$\min_{v\in V} \widehat M(v)>\tfrac{n}{2}$, we get an immediate
guarantee that $G$ is not a tree. Moreover, if $\min_{v\in V} M(v)$ is
bounded away from $\tfrac{n}{2}$, say if
$\min_{v\in V} M(v)>\tfrac23 n$, then
$\min_{v\in V} \widehat M(v)>\tfrac{n}{2}$ with high probability if
$m$ is sufficiently large. 
\begin{lem}\label{lem:treecomp} 
 For any $v\in V$ it holds that if $ M(v)>\tfrac23n$ then 
 $$
 \P(\widehat M(v)\leq \tfrac{n}{2})\;\leq\;\,e^{-\tfrac{m}{18}\tfrac{n-1}{n-m}}\;\leq\;e^{-\tfrac{m}{18}}.
 $$	
 Thus, by the union bound,
 $$
 \P(\exists v \mbox{ s.t. }M(v)>\tfrac{2n}{3}\mbox{ and }\widehat M(v)\leq \tfrac{n}{2})\;\leq\;n e^{-\tfrac{m}{18}}.
 $$
 \end{lem}
 \begin{proof}
   To estimate $\widehat M(v)$ we sampled $m=|W|$ nodes without
   replacement from $\ell=\deg(v)$ buckets of size $|C_1|,\ldots,|C_\ell|$
   where $\cC^{(v)}=\{C_1,\ldots,C_\ell\}$ (for simplicity we assume $v$
   cannot be part of the sample but the proof can be adjusted). By
   assumption the size of the maximal component, say $C$, is at least
   $\tfrac23 n$. Moreover, the distribution of $|W\cap C|$ is
   hypergeometric.
 We use the classical tail inequality 
 of \cite{Hoe63} for the hypergeometric distribution.
Let $N=n-1$ be the size of the population (nodes in
   $V\setminus \{v\}$) and let $\{x_1,\ldots,x_N\}\in \{0,1\}^N$ be
   such that $s$ of them are 0 and $N-s$ are 1. This is an indicator
   of the fact that a particular node lies \emph{outside} of the big
   component $C$. By assumption $s>\tfrac 23 (N+1)$ and so
	$$\mu\;:=\;\frac 1N\sum_{i=1}^N x_i\;=\;\frac{N-s}{N}\;<\;\frac13.$$
	Consider a sample $(I_1,\ldots,I_m)$ drawn without replacement from $\{1,\ldots,N\}$ and let $S_m=\tfrac1m\sum_{i=1}^m x_{I_i}$. 
By \cite{Hoe63} (see also Corollary~1.1 in \cite{serfling1974probability}),
\begin{equation}\label{eq:hyper2}
	\P(S_m\geq \tfrac{m}{2})\;=\;\P(S_m-m\mu\geq m(\tfrac{1}{2}-\mu))\;\leq\;\exp\left\{-2m(\tfrac{1}{2}-\mu)^2\tfrac{n-1}{n-m}\right\}.	
\end{equation}
We thus have
$$
\P(\widehat M(v)\leq \tfrac{n}{2})\;=\; \P(|C\cap W|\leq \tfrac{m}{2})\;=\;\P(S_m>\tfrac{m}{2})\;\leq\; \P(S_m\geq \tfrac{m}{2}).
$$
The last expression can be bounded using \eqref{eq:hyper2}. Since, $\mu<1/3$, we finally get
$$
\P(\widehat M(v)\leq \tfrac{n}{2})\;\leq \;  \exp\left\{-\tfrac{m}{18}\tfrac{n-1}{n-m}\right\}.
$$
\end{proof}

\subsection{Descending into sub-components}\label{sec:descending}

After completing steps (S1)-(S3), our procedure finds $v^*$, which optimizes $\widehat M(v)$ over $v\in V$, and the corresponding components $C_1, \ldots, C_\ell$. If $\widehat M( v^*)> n/2$, it stops with a guarantee that $G$ cannot be a tree (like the red node in Figure~\ref{fig:levels}). If $\widehat M(v^*)\leq  n/2$ it  descends into the connected components of $G\setminus \{v^*\}$. By this we mean that, for every $C\in \cC^{(v^*)}$, we apply our procedure to the smaller matrix $\Sigma_{B,B}$ with $B=C\cup \{v^*\}$. For each of the components $B=C\cup \{v^*\}$ for $C\in \cC^{(v^*)}$  we first check if $|B|\leq m$. If yes (the green nodes in Figure~\ref{fig:levels}), we simply query the whole matrix $\Sigma_{B,B}$, invert it, and identify the underlying subgraph directly. On the other hand, if $|B|>m$ , we run on it Algorithm~\ref{algo2} proceeding recursively. The whole procedure is outlined in Algorithm~\ref{algo3}. 

To prove correctness of this approach we argue that the structure of the underlying induced subgraph $G_{B}$ can be still directly read from the submatrix $\Sigma_{B,B}$. We formulate this result in greater generality than what we need now. This is a matrix-algebraic version of the main result in \cite{frydenberg1990marginalization}.
\begin{lem}\label{lem:marg}Let $G$ be any graph and let $C$ be one of the connected components of $G\setminus S$. Denote $B=C\cup S$ and assume $\Sigma\in M(G)$. Then $\Sigma_{B,B}\in M(G_B)$ if and only if for every pair $i,j\in S$,  not connected by an edge, $(\Sigma_{B,B})^{-1}_{ij}=0$. In particular, if $S$ is a clique then $\Sigma_{B,B}\in M(G_B)$.
\end{lem}
\begin{proof}
The right implication follows from the definition. For the left implication, denote $A=V\setminus B$ and $K=\Sigma^{-1}$.	 The standard formula for block matrix inversion gives (see Section~0.7.3, \cite{horn2012matrix})
\begin{equation}\label{eq:KBBinv}
(\Sigma_{B,B})^{-1}\;=\;K_{B,B}-K_{B,A}K_{A,A}^{-1}K_{A,B}.	
\end{equation}
We need to show that for any two $i,j\in B$ that are not connected by an edge, the corresponding entry is zero. If $i,j\in S$, this follows by assumption, so assume that one of them, say $i$, does not lie in $S$. We have
$$
(\Sigma_{B,B})^{-1}_{ij}\;=\;K_{ij}-K_{i,A}K_{A,A}^{-1}K_{A,j}.
$$
Since $i,j$ are not connected by an edge and $\Sigma\in M(G)$, we have $K_{ij}=0$. Moreover, since $i\in C$, it is also not connected to any edge in $A$ and so $K_{i,A}=0$. This completes the argument.
\end{proof}


Our testing procedure is performed by running
\texttt{TestTree}$(G,\Sigma,m)$; see Algorithm~\ref{algo3}. The
following result bounds the query complexity of this procedure and it
shows that it never breaks if $G$
is a tree. 
Even if $G$ is not a tree, the algorithm always concludes correctly.
The query complexity depends on $\Delta$, the maximal degree of
$G$. Note that the algorithm does not need to know $\Delta$ though its 
running time guarantee gets worse for large $\Delta$.

\begin{thm}\label{th:main1}
Let $G$ be a connected graph and let $\Sigma\in M^1(G)$. Fix $\epsilon<1$ and let 
\begin{equation}\label{eq:mtree}
m\;= \; \left\lceil 18\log\left(\frac{5n^2}{\epsilon}\log(n)\right) \right\rceil	
\end{equation} 
be a parameter of Algorithm~\ref{algo3}.  
The algorithm correctly identifies whether $G$ is a tree.
Moreover, with probability at least $1-\epsilon$, 
Algorithm~\ref{algo3}  runs with total query complexity	
$$
\cO(n\log(n)(m\min\{m,\Delta\}+\Delta)=\cO(n\log(n) (\log(n/\epsilon)\min\{\log(n/\epsilon),\Delta\}+\Delta))
$$
where $\Delta$ is the maximum degree of any vertex in $G$.
\end{thm}

The upper bound shows that, whenever $\Delta =O(n^{1-\gamma})$ for
some $\gamma >0$, the algorithm requires sub-quadratic query complexity.
In particular, when the maximum degree is bounded, the query
complexity is quasi-linear. It is not difficult to see that, without
any bound on $\Delta$, one cannot hope for nontrivial query
complexity. For example, in order to test whether $G$ is a star of
degree $n-1$ or it is a star with one extra edge added, any algorithm
needs $\Omega(n^2)$ query complexity. Instead of giving a formal
statement, we refer the reader to \cite{lugosi2021learning} for
related arguments in the context of structure learning.

\begin{proof}
By Lemma~\ref{lem:compl2}, Algorithm~\ref{algo2} applied initially to
  the whole graph $G$, has query complexity
  $\cO(n(m\min\{m,\Delta\}+\Delta))$. In the second step, we run the same
  procedure on each of the components $C\cup \{v^*\}$ separately (with
  the same $m$) with complexity $\cO((|C|+1)(m\min\{m,\Delta\}+\Delta))$. Since
  $\sum_C (|C|+1)\leq 2n$, the total query complexity adds up giving
  again $\cO(n(m\min\{m,\Delta\}+\Delta))$. Note that some components can be
  small (less than $m$) but this does not affect our upper bound. The
  same argument can be now applied to each level of the recursion
  tree, where we simply bound the number of components in each level
  by $n$.
	
	It remains to control the number of times this procedure is executed on subsequent sub-components. Referring to the underlying algorithm tree (like in Figure~\ref{fig:levels}), this corresponds to the number of levels in this tree and the number of components in each level. 	
	
	Let $\cE_\ell$ be the event that all components in the $\ell$-th level of the recursion tree have size bounded by $(\tfrac23)^\ell n$. Let $$\ell^*\;:=\;\left\lceil \frac{\log\left(\frac{n}{m}\right)}{\log\left(\frac{3}{2}\right)}\right\rceil\;\leq\;5\log\left(\frac{n}{m}\right).$$
	Under $\cE_{\ell^*}$, each component in the $\ell^*$-th level of the tree has size bounded by $m$. After this, one more run of the algorithm will give a definite answer to whether $G$ is a tree or not. The total query complexity is then $\cO((\ell^*+1)n (m\min\{m,\Delta\}+\Delta))=\cO(n\log(n) (m\min\{m,\Delta\}+\Delta))$ (which is the claimed complexity). 
	
	We now show that the event $\cE_{\ell^*}$ holds with probability at least $1-\epsilon$. The probability of the complement can be bounded by the probability that in at least one instance the central vertex $v^*$ output by Algorithm~\ref{algo2} does not give a balanced split of the corresponding component. Note that the number of components at each level is bounded by $n$. Moreover, using Lemma~\ref{lem:treecomp}, in each component $C$
	$$\P\left(\exists v\in C \mbox{ s.t. } M(v)>\tfrac{2|C|}{3}\mbox{ and }\widehat M(v)\leq \tfrac{|C|}{2}\right)\;\leq\; |C|e^{-\tfrac{m}{18}}.$$

	By the union bound, the probability that in at least one call we do not get a balanced split can be bounded by 
	$$
	n^2 \ell^* e^{-m/18}\;\leq\; n^2  {5\log\left(\frac{n}{m}\right)} \frac{\epsilon}{5 n^2 \log(n)}\;\leq\; \frac{\log\left(\frac{n}{m}\right)}{\log(n)}\epsilon\;\leq\;\epsilon,
	$$ 	
	which concludes the proof.
\end{proof}

\begin{algorithm}
Input: an oracle for $\Sigma$, $m\in \N$\;
\If{$|V|\leq m$}{run a direct test}
\Else{run \texttt{FindBalancedPartitionTree}$(\Sigma,m)$ to get an optimal separator $v^*$ and the corresponding components $C_1,\ldots,C_\ell$\;
\For{$i=1,\ldots,\ell$}{
$B:=C_i\cup \{v^*\}$\;
run ${\tt TestTree}(\Sigma_{B,B},m)$}}
\label{algo3}
\caption{${\tt TestTree}(\Sigma,m)$}
\end{algorithm}

\section{Testing small separation numbers}\label{sec:treewidth}

In this section we generalize the procedure for testing trees to a significantly richer class
of graphs, characterized by their separation number,
as given in \eqref{eq:test}. 
The separation
number $\sn(G)$ is defined formally as the smallest integer $s$ such
that for every subset $W\subseteq V$, with $|W|\geq k+2$, there is a partition of $W$ into
three sets $S,A,A'$, such that $|S|\leq k$,  $A$ and $A'$ are non-empty,
$\max\{|A|,|A'|\}\leq \tfrac23 |W|$ and $S$ separates $A$ and $A'$ in
the subgraph $G_W$. Such a separator is called a
$(\tfrac23,k)$-separator of $W$. In an analogous way we define
$(\alpha,k)$-separators for any $\alpha\in [\tfrac23,1)$.

The value of the separation number reveals fundamental structural properties of the graph, important
for understanding the global dependence structure in graphical models.
Also, this notion is closely related to 
another fundamental parameter of the graph, the \emph{treewidth $\tw(G)$}; see \cite{robertson1986graph} or \cite{bodlaender1998partial} for more details. 
Indeed, separation number and treewidth are within constant factors of each other:
\begin{prop}[\cite{dvovrak2019treewidth}]
	For every graph,  $\sn(G)\leq \tw(G)+1 \leq 15\sn(G)$. 
\end{prop}

Conceptually, our testing algorithm is analogous to the one presented
for trees. Thus, the rest of this section is organized in a similar
way as Section~\ref{sec:tree}. In Section~\ref{sec:central2} we
provide a procedure to efficiently find a small balanced
separator. Having such a separator $S$, the algorithm descends in the
components of $G\setminus S$. This is explained in
Section~\ref{sec:descending2}. Then the procedure proceeds
recursively and in Sections~\ref{sec:marg} and \ref{sec:cond} we consider two ways this can happen.

\subsection{Finding a balanced separator}\label{sec:central2}

In order to identify a small balanced separator in a graph $G$, we adopt
a randomized approach similar to that used for trees but with stronger genericity conditions. Our strategy
involves sampling a subset of vertices, denoted as $W$, and then
searching for a compact balanced separator of $W$. The approach
developed in \cite{feige2006finding} provides a basis for asserting
that even for relatively small values of $m=|W|$, a balanced separator
of $W$ will, with high probability, serve as a balanced separator for
the entire graph. However, it is important to note a critical
distinction from the tree case: we cannot perform an exhaustive search
across all potential separators of size $\leq k$ because this space is
too large. For this, we employ the techniques of
\cite{lugosi2021learning}, which leverage the algebraic structure of
the model to identify a minimal separator for $W$ with small query
complexity. Our approach remains within the subquadratic computational
budget. We now describe in detail how this is performed.


To describe our procedure to find a balanced separator, first note that, under strong faithfulness: if $|S|=r$ and $\rk(\Sigma_{A\cup S,A'\cup S})=r$ then $S$ is a minimal separator of $A,A'$. To find a balanced separator of $W$ we proceed as follows; see Algorithm~\ref{alg:sep}. First, search exhaustively through all partitions of $W$ into sets $A, A'$ with $\max\{|A|,|A'|\}\leq \tfrac23|W|$.  For each such split $A/A'$ compute the rank of $\Sigma_{A,A'}$. If $\sn(G)\leq k$ then such minimal rank needs to be less than or equal to $k$, which we can use as the early detection for $\sn(G)>k$. Take any split $A/A'$ that minimizes this rank, say ${\rm rank}(\Sigma_{A,A'})=r\leq k$. 

Now, if $\Sigma$ is $k$-strongly faithful, Algorithm~\ref{alg:sep0} finds a minimal separator of $A$ and $A'$. It first checks for all $v\in V$ whether $\rk(\Sigma_{A\cup\{v\},A'\cup \{v\}})=r$, which is equivalent with $v$ being an element in some minimal separator of $A,A'$; c.f. Lemma~3 in \cite{lugosi2021learning}. After identifying the set $U$ of all nodes that lie in some minimal separator of $A$, $A'$, we proceed to find a minimal separator. This is done by picking any element $v_0$ in $U$, fixing $S=\{v_0\}$, and then adding nodes from $U$ to $S$ one by one, at each step making sure that $S$ is part of a minimal separator of $A,A'$; here again we use Lemma~3 in \cite{lugosi2021learning} and simply check if $\rk(\Sigma_{A\cup S,A'\cup S})=r$. Because, $\Sigma$ is $k$-strongly faithful, this procedure concludes with $|S|=r$, a $(\tfrac23,k)$-separator $S$ of $W$.
   
   \begin{algorithm}
\label{alg:sep0}
$U \gets \emptyset$\;
$r = {\rm rank}(\Sigma_{A,A'})$\;
Let $V$ be the index set of the rows of $\Sigma$\;
\ForAll{$v\in V$}{\If{${\rm rank}(\Sigma_{Av,Bv})=r$}{$U\gets U\cup \{v\}$\;}}
$S\gets \{v_0\}$ for some $v_0\in U$\;
\ForAll{$u\in U\setminus \{v_0\}$}{\uIf{${\rm rank}(\Sigma_{ASu,BSu})=r$}{
    $S\gets S\cup \{u\}$ \;
  } }
    \Return{ $S$\;}
\caption{${\tt ABSeparator}(\Sigma,A,A')$}
\end{algorithm}

  After finding a $(\tfrac23,k)$-separator $S$ of $W$ we would like to
  argue that $S$ is also an $(\alpha,k)$-separator of the entire node set $V$ for some
  $\alpha\in [\tfrac23,1)$. To determine the size of $W$ that allows
  us to draw such a conclusion, we follow the discussion in Section~4
  in \cite{lugosi2021learning};  a key tool is from
  \cite{feige2006finding} who bound the {\sc vc} dimension of the
  class of sets of vertices forming the connected components of a
  graph obtained by removing $k$ arbitrary vertices.

  \begin{thm}\label{thm:separator}
  	Fix $0<\delta,\delta'<\tfrac13$. Suppose that $W\subseteq V$ is obtained by sampling $m$ vertices from $V$ uniformly at random, with replacement, where $m$ satisfies
  	\begin{equation}\label{eq:boundW}
 m \;\;\geq\;\; \max\left(\frac{110k}{\delta ^2}\log\!\left(\frac{88k}{\delta ^2}\right), \frac{2}{\delta ^2}\log\!\left(\frac{2}{\delta'}\right) \right)~.	
\end{equation}
Then, with probability at least $1-\delta'$,  every $(\tfrac23,k)$-separator of $W$ is a $(\tfrac23+\delta,k)$-separator of $V$. 
  \end{thm}
  \begin{proof}
A set $W \subseteq V$ is a \emph{$\delta$-sample} for $(G,k)$ if for all sets $S\subseteq V$ with $|S|\leq k$ and all $C\in \cC^\cS$, 
\begin{equation}\label{eq:dsample}
\frac{|C|}{|V|}-\delta\;\leq\;\frac{|W\cap C|}{|W|}\;\leq\; \frac{|C|}{|V|}+\delta.  	
\end{equation}
In other words, $W$ allows to accurately estimate relative sizes of all connected components, exactly as in the tree case; c.f. Definition~3.2 in \cite{feige2006finding}.  By Lemma~3.3 in \cite{feige2006finding} every $(\tfrac23,k)$-separator of a $\delta$-sample $W$ is a $(\tfrac23+\delta,k)$ separator of the whole graph. Thus, to conclude the result, we need to show that with probability $\geq 1-\delta'$, $W$ is a $\delta$-sample if $m=|W|$ satisfies \eqref{eq:boundW}. This can be done by combining Theorem~22 and Lemma~23 in \cite{lugosi2021learning}. 
\end{proof}

\subsection{Descending into sub-components}\label{sec:descending2}

Our procedure begins by testing equation \eqref{eq:test} for a given $k\geq 1$. The input is provided by an oracle on $\Sigma$. If the cardinality of $V$ is less than or equal to $k$, there is nothing to verify, and the procedure ends. Otherwise, the procedure attempts to find a small balanced separator $S$ of $G$ using Algorithm~\ref{alg:sep}, as outlined in Section~\ref{sec:central2}. If no such small balanced separator can be found, the procedure halts. If a balanced separator is found, we run \texttt{Components}$(\Sigma^{(S)},[n])$ in Algorithm~\ref{algo1} to identify the connected components of $G\setminus S$.

The algorithm then descends into the components and applies the same procedure to each subset of nodes. We outline two methods of descent. The first is both conceptually and computationally simpler. It guarantees a correct answer when the graph $G$ is decomposable, but may sometimes halt inconclusively for certain non-decomposable graphs. In such cases, we execute our second algorithm, which always provides a bound on the separation number. These two methods of descent are referred to as the marginal descent (MD) and the conditional descent (CD):
\begin{enumerate}
	\item [(MD)] For each $C\in \cC^{(S)}$, we apply our procedure to the smaller matrix $\Sigma_{B,B}$, where $B=C\cup S$. 
\item [(CD)] For each $C\in \cC^{(S)}$, we apply our procedure to the smaller matrix $\Sigma^{(S)}_{C,C}$. 
\end{enumerate}
The procedure for testing small separation number is executed by running either Algorithm~\ref{algomainmarg}, or Algorithm~\ref{algomaincond}. As we will demonstrate, both procedures yield significant insights.

\subsection{Marginal descent}\label{sec:marg}

In the marginal descent, we regress into each subset $B=C\cup S$, where $C$ is a connected component of $G\setminus S$. We then repeat the procedure on each of these smaller sets. However, justifying this step is more complex than in the case of trees. The complexity arises from how we access information about the induced subgraph $G_B$. By applying Lemma~\ref{lem:marg}, we deduce that if $S$ is a clique, then $\Sigma_{B,B}$ belongs to $M(G_B)$. Furthermore, we need to monitor how our genericity conditions change when transitioning to $\Sigma_{B,B}$. If $S$ is a clique, we also demonstrate that $\Sigma_{B,B}$ remains generic in the sense that if $\Sigma$ is strongly faithful, then $\Sigma_{B,B}$ is also strongly faithful. However, these two statements do not hold universally (refer to Proposition~\ref{prop:inB} for more details).

To discuss the graph represented by the submatrix $\Sigma_{B,B}$, we need the following definition.
\begin{defn}\label{def:closure}Suppose $S\subseteq V$, and $B=C\cup S$, where $C$ is one of the connected components of $G\setminus S$. A closure $\overline G_B$ of the induced subgraph $G_B$ of $G$ is a graph obtained from $G_B$ by adding edges between any $i,j\in S$ that are not connected by an edge in $G_B$ but for which there exists a path that (apart from its endpoints $i,j$) lies completely outside of $B$. \end{defn}
The relevance of this definition follows from the next result.
\begin{prop}\label{prop:inB}
	Suppose $|S|\leq k$, and $B=C\cup S$, where $C$ is one of the connected components of $G\setminus S$. If $\Sigma$ is $k$-faithful to $G$ then $\Sigma_{B,B}$ is $k$-faithful to $\overline G_B$, where $\overline G_B$ is the closure of $G_B$ as given in Definition~\ref{def:closure}.
\end{prop}
\begin{proof}
	We first show that $\Sigma_{B,B}\in M(\overline G_B)$. Consider the graph $\overline G$ obtained from $G$ by adding precisely the extra edges we added constructing $\overline G_B$. Since $G\subseteq \overline G$ and $\Sigma\in M(G)$, we also have $\Sigma\in M(\overline G)$. To show that $\Sigma_{B,B}\in M(\overline G_B)$, by Lemma~\ref{lem:marg}, it is enough to check that $(\Sigma_{B})^{-1}_{ij}=0$ for all $i,j\in S$ not connected by an edge in $\overline G_B$. By construction of $\overline G_B$, every path in $G$ between any such $i$ and $j$ contains a vertex in $B$. It follows that $B\setminus \{i,j\}$ separates $i$ and $j$ in $G$. We argue that this already implies that $(\Sigma_{B,B})^{-1}_{ij}=0$ as needed. Indeed, let $A=V\setminus B$ and $K=\Sigma^{-1}$. By \eqref{eq:KBBinv}
	$$
	(\Sigma_{B,B})^{-1}_{ij}\;=\;K_{ij}-K_{i,A}K_{A,A}^{-1}K_{A,j}.
	$$
	Since $i$ and $j$ are not connected, $K_{ij}=0$. Now split $A$
        into $ A_i\cup A_j$, where $A_i$ is the connected component of
        $i$ in $A$ and $A_j=A\setminus A_i$. Note that $A_j$ contains
        the connected component of $j$ and $A_i\cap A_j=\emptyset$ by
        the fact that there is no path between $i$ and $j$ in
        $A$. Then $K_{i,A}$ is supported only on $A_i$ and $K_{A,j}$
        is supported only $A_j$. Moreover $K_{A_i,A_j}=0$ and so
        $(K_{A,A})^{-1}_{A_i,A_j}=0$ too (by simple block matrix inversion). The claim follows.
		
	We now show that $\Sigma_{B,B}$ is, in addition, $k$-faithful with respect to $\overline G_B$. Indeed, suppose that for some $i,j\in B$ and $S'\subseteq B\setminus \{i,j\}$ with $|S'|\leq k$, it holds that $\Sigma_{ij}=\Sigma_{i,S'}\Sigma_{S',S'}^{-1}\Sigma_{S',j}$ (equiv. $\Sigma_{ij}^{(S')}=0$). Since $\Sigma\in M^{k}(G)$, it follows that $S'$ separates $i,j$ in $G$. We show that the same holds in $\overline G_B$. Suppose there is a path in $\overline G_B$ between $i$, $j$ that does not contain vertices from $S'$. The only possible new path must involve some of the added edges (that do not lie in $G_B$).  In other words there must exist vertices $u,v\in S$ that are not adjacent in $G$ but are in $\overline G_B$ and the edge $(u,v)$ is contained in the said path. But recall that $u,v\in S$ are joined in the construction of $\overline G_B$ only if there exists a path $P$ between them that is entirely contained outside $B$. Note that $P$ cannot contain vertices in $S'$ as it lies outside of $B$. Replacing every such potential edge $(u,v)$ with the corresponding path $P$, gives a path between $i$ and $j$ in $G$ that contains no vertices in $S'$. This leads to a contradiction because we said that no such path can exist.   \end{proof}
	

\begin{algorithm}
Input: an oracle for $\Sigma$, $k\geq 1$ (separation number to test)\;
Let $V$ be the index set of rows of $\Sigma$\; 
\If{$|V|\leq k$}{{\tt STOP}}
\Else{run \texttt{Separator}$(\Sigma)$ to get a balanced separator
  $S$, such that $|S|\leq k$ and \texttt{Components}$(\Sigma^{(S)},V)$ to get the corresponding components $C_1,\ldots,C_\ell$\;
\For{$i=1,\ldots,\ell$}{
$B:=C_i\cup S$\;
run ${\tt test.marginal}(\Sigma_{B,B},k)$}}
\label{algomainmarg}
\caption{${\tt test.marginal}(\Sigma,k)$}
\end{algorithm}

We rely on the following important result that shows that local separators found by our procedure in some component $B$, become global separators of $G$. 
\begin{lem}\label{lem:loc2globsep}
	Suppose $\Sigma$ is $k$-faithful to $G$. Let $S$ be a separator satisfying $|S|\leq k$ found at some point in the algorithm by running {\tt Separator}$(\Sigma_{B,B})$. Let $C_1,\ldots,C_\ell$ be the corresponding components obtained by running {\tt Components}$(\Sigma^{(S)}_{B,B})$. Then $S$ separates $C_1,\ldots,C_\ell$ in $G$.  
\end{lem}
Note that this result is not immediately obvious. 	By Proposition~\ref{prop:inB}, if $\Sigma\in M^k(G)$, then $\Sigma_{B,B}\in M^k(\overline G_B)$. Thus, by construction, $S$ separates $C_1,\ldots,C_\ell$ in $\overline G_B$ but the claim is about separation in $G$.
	\begin{proof}
		 Consider two components $C_i,C_j$ and suppose that $u\in C_i$, $v\in C_j$ are connected by a path $P$ in $G$ that does not cross $S$. Then, $P$ contains a subpath with one endpoint $u'$ in $C_i$, another $v'$ in $C_j$ that lies entirely outside of $B$. This however leads to contradiction, because it implies that $u',v'$ are connected in $\overline G_B$ and so, they cannot lie in two disconnected components of $\overline G_B\setminus S$.   
	\end{proof}

\begin{algorithm}
\label{alg:separator}
Pick a set $W$ by taking $m$ vertices uniformly at random with replacement, where $m$ satisfies (\ref{eq:boundW}) with $\delta=\tfrac{7}{30}$ and $\delta'=\tfrac{\epsilon}{10 n \log(n)}$ ($\epsilon\in (0,1)$ is a fixed parameter) \;
Search exhaustively through all partitions of $W$ into sets
$ A,A'$ with $|A|,|A'|\leq \frac{2}{3}|W|$, minimizing ${\rm rank}(\Sigma_{A,A'})$\;
Let $A/A'$ be any partition minimizing the rank\;
\If{${\rm rank}(\Sigma_{A,A'})>k$}{{\tt BREAK}\;}
$S\gets {\tt ABSeparator}(\Sigma,A,A')$\;
\Return{ $S$}
  \label{alg:sep}
\caption{${\tt Separator}(\Sigma,k)$}
\end{algorithm}

In our algorithm, the crucial step involves finding a balanced separator. Suppose that $\Sigma$ is $k$-strongly faithful to graph $G$. In the initial step, we identify $S$ as a minimal separator in the random sample $W$. However, in subsequent steps, we apply our algorithm to $\Sigma_{B,B}$. According to Proposition~\ref{prop:inB}, if $\Sigma$ is $k$-strongly faithful to $G$, then $\Sigma_{B,B}$ is $k$-faithful to the complement graph $\overline G_B$. However, it is important to note that $\Sigma_{B,B}$ is not necessarily $k$-strongly faithful to $\overline G_B$. This property is essential for Algorithm~\ref{alg:separator} to function correctly. As a consequence, even if $\rk(\Sigma_{A,A'})=r$ (meaning that subsets $A$ and $A'$ are separated by a subset of size $r$), Algorithm~\ref{alg:separator} might output a proper subset of a minimal separator if that separator is not contained within $B$.


\begin{defn}\label{def:goodrun}
	We say that Algorithm~\ref{algomainmarg} has a \emph{good run} if, in each call of Algorithm~\ref{alg:sep0}, the output $S$ of the algorithm is a separator of $A$ and $A'$, equivalently, $|S|=\rk(\Sigma_{A,A'})$. 
\end{defn}	
\noindent  If $\Sigma$ is $k$-strongly faithful to $G$ this definition is just saying that in each call of Algorithm~\ref{alg:sep0} the sets $A$, $A'$ are minimally separated within the current component $B$. A sufficient condition for this to happen is that, for each component $B$, $\Sigma_{BB}$ is $k$ strongly faithful to $\overline G_B$. 
\begin{lem}\label{lem:decom}
	If $G$ is a decomposable graph with $\sn(G)\leq k$  and $\Sigma$ is $k$-strongly faithful to $G$ then, for every subsequent component $B$, in the run of Algorithm~\ref{algomainmarg}, $\Sigma_{BB}$ is $k$-strongly faithful to $G_B$.
\end{lem}
\begin{proof}
Let $S$ be a minimal separator found in the first step of the procedure, and suppose we descend into $B = C \cup S$, where $C$ is one of the components in $\mathcal{C}^{(S)}$. \cite{dirac1961rigid} characterized decomposable graphs as those for which every minimal separator is a clique. Thus, in our case, $S$ is a clique, and $\overline{G}_B = G_B$.  To show that $\Sigma_{B,B}$ is $k$-strongly faithful, suppose that $\rk(\Sigma_{A,A'}) = r \leq k$ for some $A, A' \subseteq B$. Since $\Sigma \in M^{k,\circ}(G)$, $A$ and $A'$ are minimally separated in $G$ by some $S'$ with $|S'| = r$. We show that $S' \subseteq B$. Suppose that $S'$ contains a vertex $v \notin B$. By minimality of $S'$, there is a path $P$ between $A$ and $A'$ that crosses $v$ but no other element of $S'$. Let $P_1$ be the part of $P$ that leads from $A$ to $v$, and $P_2$ be the part from $v$ to $A'$. Both $P_1$ and $P_2$ contain vertices in $S$. Let $u_1$ be the first such vertex on $P_1$, and $u_2$ be the last such vertex on $P_2$ (where $u_1 = u_2$ is possible). Since $S$ forms a clique, $u_1$ and $u_2$ are connected. Thus, walking along $P_1$, jumping from $u_1$ to $u_2$, and then going to $A'$ along $P_2$ gives a path from $A$ to $A'$ with no vertices in $S'$. But this is a contradiction. 
	\end{proof}

Decomposability is a sufficient but definitely not a necessary condition for our algorithm to have good run. Also, note that there is a simple way to detect if the algorithm has good run. Simply check if the set $S$, output by Algorithm~\ref{alg:sep0}, satisfies $|S|=\rk(\Sigma_{A,A'})$. 

\begin{thm}\label{th:main3}
  Let $G$ be a connected graph and let $\Sigma$ be $k$-strongly faithful to $G$. Fix $\epsilon\in (0,1)$ and define $m$ to be minimal satisfying
  \eqref{eq:boundW} with $\delta=\tfrac{7}{30}$ and
  $\delta'=\tfrac{\epsilon}{10 n \log(n)}$. 
Suppose Algorithm~\ref{algomainmarg} has good run. If Algorithm~\ref{algomainmarg} terminates, then $\sn(G)
\leq {2k}$. If it breaks, then $\sn(G)
> {\tfrac{2}{3}k}$.
Moreover, with probability at
  least $1-\epsilon$, it runs with total
  query complexity
$$
\cO(n\log(n)\max\{m,k\Delta\}) = O(nk\Delta\log(n) + n\log(n) k\log(k)
+n\log^2(n) + n\log(n/\epsilon))~.
$$
\end{thm}

For the query complexity analysis, the following result is useful.
\begin{lem}\label{prop:runsep}
	Running {\tt Separator}$(\Sigma_{V,V},k)$ in Algorithm~\ref{alg:sep} takes query complexity $\cO(|V|m )$. 
\end{lem}
\begin{proof}
	Finding a balanced partition $A,A'$ of $W\subseteq V$ by
        running {\tt Separator}$(\Sigma_{V,V},k)$ is achieved by
        exhaustively searching all $<2^m$ balanced partitions of $W$
        and computing the rank of the associated matrixes
        $\Sigma_{A,A'}$, which gives query complexity $\cO(m^2)$ needed
        to query $\Sigma_{W,W}$. Then running {\tt
          ABSeparator}$(\Sigma_{V,V},A,A')$ in Algorithm~\ref{alg:sep0}
        takes $\cO(m |V|)$  for the first \textbf{forall} loop. In the
        second \textbf{forall} loop we need to query $\cO(k^2)$ new
        entries but this will not affect the order of magnitude of the      total complexity because both $|V|$ and $m$ are larger than $k$. The total cost is $\cO(|V|m)$.
\end{proof}

\begin{proof}[Proof of Theorem~\ref{th:main3}]
	The proof is split into three parts. We first prove the
  bound on the total query complexity. Then we prove that the claimed bound for the query complexity holds with high
  probability. Finally, we show how one can make the conclusions on the size of $\sn(G)$ depending on
whether the algorithm breaks or terminates. Throughout we assume that Algorithm~\ref{algomainmarg} has good run, or equivalently, in each call {\tt ABSeparator}$(\Sigma_{V,V},A,A')$, the output $S$ satisfies $|S|=\rk(\Sigma_{A,A'})$.

\textbf{Query complexity bound:}
Consider the first step of our procedure assuming $|V|=n>k$. By Lemma~\ref{prop:runsep}, {\ttfamily Separator}$(\Sigma,k)$ takes query complexity $\cO(nm)$. If the procedure does not break then a $(\tfrac23,k)$-separator $S_1$ of $W$ is found. Run  {\ttfamily Components}$(\Sigma^{(S_1)},V\setminus S_1)$. Note that querying the entries of
$\Sigma^{(S_1)}$ has some initial cost $|S_1|^2$ needed to get
$\Sigma_{S_1,S_1}$ (but this submatrix was queried already in the earlier step) and for each entry $\Sigma^{(S_1)}_{ij}$ the cost is
$1+2|S_1|=\cO(|S_1|)$ to query $\Sigma_{ij}$ and the vectors
$\Sigma_{i,S}$, $\Sigma_{j,S}$. 
Using Lemma~\ref{lem:comp}, we conclude that the query complexity of getting the components is $\cO(|S_1|n \Delta)=\cO(kn\Delta)$. Thus, the total query complexity for the first step  is $\cO(n\max\{m,k\Delta\})$. 

In the second level of the recursion tree we need to take into account
the cost for each $B=C\cup S_1$, where $C$ runs over all the connected
components of $G\setminus S_1$. By Lemma~\ref{prop:runsep}, the cost
of running {\ttfamily Separator}$(\Sigma_{BB},k)$ on a single
component $B=C\cup S_1$ is  $\cO(|C|m)$. Thus, the total query
complexity of running this on all connected components is
$\cO(nm)$. Let $S$ be the separator found in the component $B=C\cup
S_1$. The cost of running {\ttfamily
  Components}$(\Sigma^{(S)}_{B,B},B\setminus S)$ in this component is
$\cO(k|C|\Delta)$, which is additive across the components. Overall,
the total cost in the second level is $\cO( n\max\{m,k\Delta\})$ and
the same query complexity holds for every level $\ell$ in the recursion tree. 

	To get the total query complexity, it remains to control the number of times this procedure is executed on subsequent sub-components. Referring to the underlying algorithm tree (like in Figure~\ref{fig:levels}), this corresponds to the number of levels in this tree and the number of components in each level. 	
	
	Let $\cE_\ell$ be the event that all components in the $\ell$-th level of the recursion tree have size bounded by $(\tfrac{2}{3}+\delta)^\ell n=(\tfrac{9}{10})^\ell n$. Let 
	\begin{equation}\label{eq:ellstar}
		\ell^*\;:=\;\left\lceil \frac{\log\left(\frac{n}{k}\right)}{\log\left(\frac{10}{9}\right)}\right\rceil\;\leq\;10\log\left(\frac{n}{k}\right).
	\end{equation}
	Under $\cE_{\ell^*}$, each component in the $\ell^*$-th level of the tree has size bounded by $k$. After this, one more run of the algorithm will stop.  Summing over all $\ell^*+1$ levels, the total query complexity is then $\cO({\ell^*}n\max\{m,k\Delta\})$, which using \eqref{eq:ellstar} gives the claimed complexity.

\textbf{Bound on the probability of getting the claimed query complexity:}	
	We now show that the event $\cE_{\ell^*}$ holds with probability at least $1-\epsilon$. The probability of the complement of $\cE_{\ell^*}$ can be bounded by the probability that in at least one instance the separator $S$  output by Algorithm~\ref{alg:sep0} is not a $(\tfrac{9}{10},k)$-separator of the whole component. By Theorem~\ref{thm:separator} (with $\delta=\tfrac{7}{30}$ and $\delta'=\tfrac{\epsilon}{10 n \log(n)}$) this happens with probability $\leq \delta'$. The number of components on which we run this algorithm is bounded by $n\ell^*$ and so, by the union bound, this probability can be bounded by 
	$$
	n \ell^* \delta'\;\leq\;n 10 \log\left(\frac{n}{k}\right)\frac{\epsilon}{10 n \log(n)}\;\leq\;\epsilon.
	$$ 	

	\textbf{Correctness of the algorithm:} Suppose that
        the algorithm terminated without breaking. We claim that in
        this case $\sn(G)\leq 2k$. To show this, let $W$ be an arbitrary
        subset of $[n]$ of size $k+2$. Consider the recursion tree $\cT$ given by
        the output of the algorithm, where each inner node of
        $\cT$ represents the separator output at this stage of
        the whole procedure. In particular, the root represents the
        output $S_1$ of {\tt Separator}$(\Sigma,k)$. The leaves of $\cT$
        represent the remaining subsets $L_1,\ldots,L_t$ each of size
        $\leq k$. We show that there exists a $(\tfrac23,2k)$-separator of $W$.
       
        Start with the root $S_1$ and recall that $|S_1|\leq k$ by construction. If $|C\cap W|\leq \tfrac{2}{3}|W|$ for all $C\in \cC^{(S_1)}$ then, by Lemma~26 in \cite{lugosi2021learning}, $S_1$ is a $(\tfrac23,k)$-separator of $W$ and we are done. Suppose then that there exists a component $C_1$ with $|C_1\cap W|> \tfrac{2}{3}|W|$. Denote $B_1=S_1\cup C_1$ and $A_1=V\setminus B_1$. We have $|B_1\cap W|> \tfrac{2}{3}|W|$, $|A_1\cap W|< \tfrac{1}{3}|W|$. Let $S_2$ be the separator found by the algorithm in $B_1$. This separates $B_1$ into components of $\overline G_{B_1}\setminus S_2$. Again, there is at most one component containing more than $\tfrac{2}{3}|W|$ elements of $|W|$. This process must stop at some point, that is, there is a $B_t=S_t\cup C_t$ with $t\geq 1$  containing more than  $\tfrac{2}{3}|W|$ elements of $|W|$ such that either $B_t$ is one of the leaves of $\cT$ (set $S_{t+1}=\emptyset$), or all the components of $\overline G_{B_t}\setminus S_{t+1}$ contain $\leq \tfrac{2}{3}|W|$ elements of $|W|$. Note that $A_t=V\setminus B_t$ contains $< \tfrac{1}{3}|W|$ elements of $W$. Moreover, by Lemma~\ref{lem:loc2globsep}, the set $S_t\cup S_{t+1}$ (of size $\leq 2k$) separates $A_t$ and all the components of $\overline G_{B_t}\setminus S_{t+1}$ from each other. In consequence, it induces the split of $W$ into $\tfrac23$-balanced sets.

 Now suppose that the algorithm broke after reaching a component $B=S\cup C$, that is, $|S|\leq k$ but it was impossible to find a small balanced separator in $\overline G_B$. We claim that $\sn(G)>\tfrac23 k$. Suppose on the other hand that $\sn(G)\leq \tfrac23 k$. Then, there exists a subset $S'\subseteq B$ of size $\leq \tfrac23 k$ which splits the induced graph $G_B$ into two parts $A',B'$ such that $A'\cup B'\cup S'=B$ and $\max\{|A'|,|B'|\}\leq \tfrac23|B|$. Since $S'$ cannot separate $A'$ from $B'$ in $\overline G_B$, we get that $A'\cap S\neq \emptyset$, $B'\cap S\neq \emptyset$. Without loss of generality     assume $|A'\cap S|\leq |B'\cap S|$ and so $|A'\cap S|\leq k/3$. The set $\tilde S=S'\cup (A'\cap S)$ satisfies $|\tilde S|\leq k$ and it splits $B$ in $G_B$ into $A'\setminus S$, $B'$ both of size $\leq \tfrac23 |B|$.
\end{proof}
	

Under a good run, one may use Algorithm~\ref{algomainmarg} and Theorem
\ref{th:main3} to estimate the separation number of a graph.
One can run the algorithm repeatedly with parameters $k=1,2,\ldots$
and $\epsilon/k^2$
until the the first time it terminates. Denoting this value by $k_0$,
Theorem \ref{th:main3} guarantees that $\sn(G)\le 2k_0$. Moreover,
since the algorithm breaks for $k_0-1$, we also know that $\sn(G) >
2(k_0-1)/3$.
Hence, we have a guarantee that $\sn(G)\in (2(k_0-1)/3, 2k_0]$.
With probability at least $1-\epsilon$, the total query complexity is $\cO(n\cdot\sn(G)\log(n)\max\{m,\sn(G)\Delta\})$.

	It is important to note that, in the special case when $G$ is decomposable,  Algorithm~\ref{algomainmarg} always gives a definite answer. In this case we recover the ideal situation that we encountered for trees. 

\begin{thm}\label{th:decomp}
	With the same assumptions as in Theorem~\ref{th:main3} assume in addition that $G$ is decomposable. Then Algorithm~\ref{algomainmarg} terminates if and only if  $\sn(G)
\leq \textcolor{blue}{k}$.\end{thm}
	\begin{proof}
	By Lemma~\ref{lem:decom}, at each step $B=C\cup S$ of the algorithm, $\Sigma_{BB}$ is $k$-strongly faithful to $G_B$. If $\sn(G)\leq k$ then at each step we are able to find a small balanced separator and so the algorithm never breaks.    On the other hand, if $\sn(G)>k$ then $G$ contains a clique of size $>k$. This implies that the algorithm must break at some point. 
	\end{proof}

Our focus is on the query complexity but we can also bound the
computational complexity of the algorithm. The total running time 
is easily seen to be $$O(n\log(n)2^m\max(m,k\Delta)).$$ The extra
exponential factor is due to the fact that, at each step, the
algorithm performs an exhaustive search over all  partitions of the
set $W$ of size $m$. Since $m=O(k\log k)$, this factor does not affect
the dependence on $n$. However, as it is usual in graph algorithms,
the dependence on $k$ (i.e., the treewidth) is exponential.

\subsection{Conditional descent}\label{sec:cond}

Although the marginal descent is theoretically appealing, it may
happen that $G$ is such that Algorithm~\ref{algomainmarg} has a good run with small probability. In this case we propose an alternative.

Recall that $M(G)$, defined in
\eqref{eq:defMG}, denotes the set of all covariance matrices $\Sigma$
such that $(\Sigma^{-1})_{ij}=0$ if $i$ and $j$ are not connected by
an edge in $G$. We say that $\Sigma$ is Markov to $G$. The next lemma shows that $\Sigma^{(S)}_{C,C}$ is Markov with respect to $G_C$. 
\begin{lem}\label{lem:lugosi2021learning}
	If $S$ is a separator of $G$ and $C\in \cC^{(S)}$ then $\Sigma^{(S)}_{C,C}=((\Sigma^{-1})_{C,C})^{-1}\in M(G_C)$. Moreover, if $\Sigma$ is $\tau$-strongly faithful to $G$, with $\tau=|S|+k$, then $\Sigma^{(S)}_{C,C}$ is $k$-strongly faithful to $G_C$.\end{lem}
\begin{proof}
	The first part is well-known; see, e.g.,  Lemma 29, \cite{lugosi2021learning}. For the second part, by the Guttman rank additivity formula, 
	$$
	\rk(\Sigma_{A\cup S,A'\cup S})=r+|S|
	$$
	and so, since $\Sigma$ is $\tau$-strongly faithful to $G$, the smallest separator of $A\cup S$ and $A'\cup S$ in $G$ has size $r+|S|$. Note that this separator must contain $S$ and so it is of the form $S\cup S'$ for some $S'$ of size $r$, $S\cap S'=\emptyset$. We now show that $S'$ separates $A$ and $A'$ in $G_C$. Indeed, every path in $G$ between $A$ and $A'$ must cross $S\cup S'$ and so every path between $A$ and $A'$ in $G_C$ must cross $S'$. We also have that $S'$ is minimal with such property for otherwise $S\cup S'$ would not be a minimal separator of $A\cup S$ and $A'\cup S$. In particular, $S'\subseteq C$. \end{proof}
\begin{rem}\label{rem:whynotM2}
	 A similar argument shows that we could get access to the desirable subgraph $G_B$  by computing $\Sigma_{B,B}^{(V\setminus B)}$. Computing this matrix is, however, beyond our query budget.
\end{rem}

One important difference in the conditional descent compared to the marginal descent is that, in each descent, we need to keep
      track of all separators in the previous steps. Suppose
      $S_1,\ldots,S_t$ are all these separating sets and
      let $$\overline S\;=\;\bigcup_{i=1}^t S_i.$$ To check separation
      of $i,j$ by $S_t$ in the current step, we actually need to check
      if $\Sigma^{(\overline S)}_{ij}=0$. Nevertheless, we still have a
      tight control over the query complexity.

\begin{algorithm}
Input: an oracle for $\Sigma$, $k\geq 1$ (separation number to test)\;
Let $V$ be the index set of rows of $\Sigma$\;
\If{$|V|\leq k$}{{\tt STOP}}
\Else{run \texttt{Separator}$(\Sigma,k)$ to get a balanced separator $S$, s.t. $|S|\leq k$ and \texttt{Components}$(\Sigma^{(S)},V\setminus S)$ to get the corresponding components $C_1,\ldots,C_\ell$\;
\For{$i=1,\ldots,\ell$}{
run ${\tt test.conditional}(\Sigma_{C_i,C_i}^{(S)},k)$}}
\label{algomaincond}
\caption{${\tt test.conditional}(\Sigma,k)$}
\end{algorithm}

Algorithm~\ref{algomaincond} offers a procedure to test the separation number for a general graph $G$. It is analogous to Algorithm~\ref{algomaincond} with the only difference in how the algorithm descends to smaller components. To carry out a formal analysis of Algorithm~\ref{algomaincond},
note that querying the entries of
$\Sigma^{(S)}_{C,C}=\Sigma_{C,C}-\Sigma_{C,S}\Sigma_{S,S}^{-1}\Sigma_{S,C}$
for $S\neq \emptyset$ has some initial cost $|S|^2$ needed to get
$\Sigma_{S,S}$ and for each entry $\Sigma^{(S)}_{ij}$ the cost is
$1+2|S|=\cO(|S|)$ to query $\Sigma_{ij}$ and the vectors
$\Sigma_{i,S}$, $\Sigma_{j,S}$.
\begin{rem}\label{rem:comp}
	The query complexity of ${\tt Components}(\Sigma^{(\overline S)}_{V,V},V)$ is $\cO(|V|\min\{\Delta,|V|\})$ if $\overline S=\emptyset$ and $\cO(|\overline S||V|\min\{\Delta,|V|\})$ otherwise.
\end{rem}

\begin{lem}\label{prop:runsep}
	Suppose $m$ satisfies (\ref{eq:boundW}) with $\delta=\tfrac{7}{30}$ and $\delta'=\tfrac{\epsilon}{10 n \log(n)}$. Running {\tt Separator}$(\Sigma^{(\overline S)}_{V,V},k)$ in Algorithm~\ref{alg:sep} takes query complexity $\cO(|V|m )$ if $\overline S=\emptyset$ and $\cO(|\overline S|(|\overline S|+|V|m))$   otherwise. 
\end{lem}
\begin{proof}
	Finding a balanced partition $A,A'$ of $W\subseteq V$ by running {\tt Separator}$(\Sigma^{(\overline S)}_{V,V},k)$ is achieved by exhaustively searching all $<2^m$ balanced partitions of $W$ and computing the rank of the associated matrixes $\Sigma_{A,A'}^{(\overline S)}$, which gives query complexity $O(|\overline S|^2+|\overline S|m^2)$. Here $|\overline S|^2$ is the price we pay for querying $\Sigma_{\overline S,\overline S}$ and then each entry of $\Sigma^{(\overline S)}$ costs us $\cO(|\overline S|)$ and we need to query $\cO(m^2)$ such entries, that is, the entries of $\Sigma^{(\overline S)}_{W,W}$.  Then running {\tt ABSeparator}$(\Sigma^{(\overline S)}_{V,V},A,A')$ in Algorithm~\ref{alg:sep0} takes $|\overline S|m |V|$  for the first \textbf{forall} loop. In the second \textbf{forall} loop we need to query $\cO(k^2)$ new entries but this will not affect the total complexity because both $|V|$ and $m$ are larger than $k$. The total cost is $\cO(|\overline S|^2+|V||\overline S|m)$.
\end{proof}

The next theorem establishes a bound for the query complexity of the whole procedure and it shows that if $\sn(G)\leq k$ then the algorithm always terminates without breaking. 

\begin{thm}\label{th:main2}
  Let $G$ be a connected graph and let $\Sigma\in M(G)$ be
  strongly faithful. Fix $\epsilon<1$ and define $m$ to be minimal satisfying
  \eqref{eq:boundW} with $\delta=\tfrac{7}{30}$ and
  $\delta'=\tfrac{\epsilon}{10 n \log(n)}$. 
{
Then, if $\sn(G)\leq k$, Algorithm~\ref{algomaincond} will never break. 
Moreover, with probability at
  least $1-\epsilon$, it runs with total
  query complexity
$$
\cO(n\log^2(n)k\max\{m,\Delta\})~
$$
and if it terminates without breaking, we conclude $\sn(G)\leq 10 k \log(\tfrac{n}{k})$.
}
\end{thm}
\begin{proof}The proof is split into three parts. We first prove the
  bound on the total query complexity. Then we prove that the claimed bound for the query complexity holds with high
  probability. Finally, we show that
  when our algorithm terminates, we are guaranteed that 
  $\sn(G)\leq 10 k \log(\tfrac{n}{k})$. Note that, by Lemma~\ref{lem:lugosi2021learning}, in each step $\Sigma^{(S)}_{CC}$ is strongly faithful to $G_C$ so the procedure {\tt SeparatorAB} correctly outputs a $(\tfrac23,k)$-separator of $C$ if only such separator exists.

\textbf{Query complexity bound:}
Consider the first step of our procedure assuming $|V|=n>k$ and $\overline S=\emptyset$. By Lemma~\ref{prop:runsep}, {\ttfamily Separator}$(\Sigma,k)$ takes query complexity $\cO(nm)$. If the procedure does not break then a $(\tfrac23,k)$-separator $S_1$ of $W$ is found. We set $\overline S=S_1$ and run  {\ttfamily Components}$(\Sigma^{(\overline S)},V\setminus \overline S)$ which, by Remark~\ref{rem:comp}, has query complexity $\cO(|\overline S|n \Delta)=\cO(kn\Delta)$. Both these steps take query complexity $\cO(n(k\Delta+m))$. 

In the second level of the recursion tree we need to take into account the cost for each of the connected components of $G\setminus \overline S$. By Lemma~\ref{prop:runsep}, the cost of running {\ttfamily Separator}$(\Sigma^{(\overline S)}_{CC},k)$ on a single component $C$ is  $\cO(|\overline S|^2+|C|m|\overline S|)=\cO(|C|mk)$. Thus, the total query complexity of running this on all connected components is $\cO(nmk)$. Let $S$ be the separator found in the component $C$. The cost of running {\ttfamily Components}$(\Sigma^{(\overline S\cup S)}_{C,C},C\setminus S)$ in this component is $\cO(k|C|\Delta)$, which is additive across the components. Overall, the total cost in the second level is $\cO(k n\max\{m,\Delta\})$.

Now consider an arbitrary $\ell$-th level. In a given component we work with a conditioning set $\overline S=S_1\cup \cdots\cup S_{\ell-1}$. Note that, to query $\Sigma_{C,C}^{\overline S}$, we need only $\cO(k|\overline S|)$ additional entries to query the matrix $\Sigma_{\overline S,\overline S}$ because the other entries were queried in the preceding steps. By similar arguments as above, in the $\ell$-th level of the recursion tree we are guaranteed that the corresponding sets $\overline S$ on which we have to condition have size $\cO(k\ell)$ and the total query complexity is 
$$\cO(\ell k n \max\{m,\Delta\})\;=\;\mbox{query complexity in the $\ell$-th level }.$$

	To get the total query complexity, it remains to control the number of times this procedure is executed on subsequent sub-components. Referring to the underlying algorithm tree (like in Figure~\ref{fig:levels}), this corresponds to the number of levels in this tree and the number of components in each level. 	
	
	Let $\cE_\ell$ be the event that all components in the $\ell$-th level of the recursion tree have size bounded by $(\tfrac{2}{3}+\delta)^\ell n=(\tfrac{9}{10})^\ell n$. Let 
	\begin{equation}\label{eq:ellstar}
		\ell^*\;:=\;\left\lceil \frac{\log\left(\frac{n}{k}\right)}{\log\left(\frac{10}{9}\right)}\right\rceil\;\leq\;10\log\left(\frac{n}{k}\right).
	\end{equation}
	Under $\cE_{\ell^*}$, each component in the $\ell^*$-th level of the tree has size bounded by $k$. After this, one more run of the algorithm will stop.  Summing over all $\ell^*+1$ levels, the total query complexity is then $\cO({\ell^*}^2kn\max\{m,\Delta\})$, which using \eqref{eq:ellstar} gives the claimed complexity.

\textbf{Bound on the probability of getting the claimed query complexity:}	
	We now show that the event $\cE_{\ell^*}$ holds with probability at least $1-\epsilon$. The probability of the complement of $\cE_{\ell^*}$ can be bounded by the probability that in at least one instance the separator $S$  output by Algorithm~\ref{alg:sep0} is not a $(\tfrac{9}{10},k)$-separator of the whole component. By Theorem~\ref{thm:separator} (with $\delta=\tfrac{7}{30}$ and $\delta'=\tfrac{\epsilon}{10 n \log(n)}$) this happens with probability $\leq \delta'$. The number of components on which we run this algorithm is bounded by $n\ell^*$ and so, by the union bound, this probability can be bounded by 
	$$
	n \ell^* \delta'\;\leq\;n 10 \log\left(\frac{n}{k}\right)\frac{\epsilon}{10 n \log(n)}\;\leq\;\epsilon.
	$$ 	

	\textbf{Correctness of the algorithm:} At each step $\Sigma^{(S)}_{CC}$ is strongly faithful to $G_C$. If the algorithm breaks
        at any stage or one of the direct checks gives a negative
        answer, then we are guaranteed that $\sn(G)>k$. Suppose that
        the algorithm terminated without breaking. We claim that in
        this case $\sn(G)\leq 10 k\log (\tfrac{n}{k})$. To show it, let $W$ be an arbitrary
        subset of $[n]$. Consider the recursion tree $\cT$ given by
        the output of our algorithm. Think about each inner node of
        $\cT$ as representing the separator output at this stage of
        the whole procedure. In particular, the root represents the
        output $S_1$ of {\tt Separator}$(\Sigma,k)$. The leaves of $\cT$
        represent the remaining subsets $L_1,\ldots,L_t$ each of size
        $\leq k$. We show that there exists a balanced separator $\overline S$ of size at most $10k\log(\tfrac{n}{k})$.
        
        Start with the root $S_1$ and recall that $|S_1|\leq k$ by construction. If $|C\cap W|\leq \tfrac{2}{3}|W|$ for all $C\in \cC^{(S_1)}$ then $S_1$ is a $(\tfrac23,k)$-separator of $W$ and we are done. Suppose then that there exists a component $C_1$ with $|C_1\cap W|> \tfrac{2}{3}|W|$. Denote  $A_1=V\setminus C_1$. We have $|A_1\cap W|< \tfrac{1}{3}|W|$. Let $S_2$ be the separator found by the algorithm in $C_1$. This separates $C_1$ into components of $G_{C_1}\setminus S_2$. Again, there is at most one component containing more than $\tfrac{2}{3}|W|$ elements of $|W|$. This process must stop at some point, that is, there is a $C_t$ with $t\geq 1$  containing more than  $\tfrac{2}{3}|W|$ elements of $|W|$ such that either $C_t$ is one of the leaves of $\cT$ (set $S_{t+1}=\emptyset$), or all the components of $ G_{C_t}\setminus S_{t+1}$ contain $\leq \tfrac{2}{3}|W|$ elements of $|W|$. Note that $A_t=V\setminus C_t$ contains $< \tfrac{1}{3}|W|$ elements of $W$. Let $\overline S=\bigcup_{l=1}^{t+1}S_l$. By construction, $S_{t+1}$ separates the components of $G_{C_t}\setminus S_{t+1}$ in $C_t$. Then $S_t\cup S_{t+1}$ separates these components in $C_{t-1}$. Proceeding recursively, we get that $\overline S$ separates all these components in the graph $G$. The fact that $|\overline S|\leq 10 k \log (\tfrac{n}{k})$ follows by \eqref{eq:ellstar}.
\end{proof}

\section{Non-gaussian case}\label{sec:ngtrees}
 
Although our paper focuses on the case when the underlying random vector is Gaussian, our results can be significantly extended. First, it is important to emphasize that zeros in the inverse of $\Sigma$ have a clear statistical interpretation even if $X$ is not Gaussian. Denote by $\rho^{ij}$ the corresponding \emph{partial} correlation between $X_i$ and $X_j$ given $X_{\setminus ij}:=X_{V\setminus \{i,j\}}$. Formally $\rho^{ij}$ is defined as the correlation of residuals $\epsilon_i$ and $\epsilon_j$ resulting from linear regression of $X_i$ on $X_{\setminus ij}$ and $X_j$ on $X_{\setminus ij}$ respectively. We get (see Section~5.1.3 in \cite{lau96})
$$
\rho^{ij}\;=\;-\frac{(\Sigma^{-1})_{ij}}{\sqrt{(\Sigma^{-1})_{ii}(\Sigma^{-1})_{jj}}}\qquad\mbox{and so}\qquad (\Sigma^{-1})_{ij}=0\quad\Longleftrightarrow\quad  \rho^{ij}=0.
$$
It may be argued that this linear conditional independence may not be interesting in general but in various specific scenarios, e.g., for elliptical distributions \cite{rossell2021dependence}, we get the following result: Suppose that $(\Sigma^{-1})_{ij}=0$ then $\E(X_i|X_{\setminus i})=\E(X_i|X_{\setminus ij})$ and $\E(X_j|X_{\setminus j})=\E(X_j|X_{\setminus ij})$. This, so called mean independence, is only slightly weaker than the classical notion of conditional independence.  

The Gaussian graphical model construction is also easily extended to
so-called nonparanormal distributions \cite{liu2012high}. In this case
the same procedures allow us to test the underlying graphical model in
a much more general distributional setting. We say that $X$ has
nonparanormal distribution if there exist monotone functions
$f_i:\R\to \R$ such that $X=(f_1(Z_1),\ldots,f_n(Z_n))$, where
$Z=(Z_1,\ldots,Z_n)$ is multivarite Gaussian. For a random vector $X$,
define the Kendall-$\tau$ coefficients $$
\tau(X_i,X_j)\;=\;\E(\indic(X_i>X_i')\indic(X_j>X_j')),
$$
where $X'$ is an independent copy of $X$. We use the classical result of \cite{kruskal1958ordinal} that links the Kendall-$\tau$ coefficient $\tau(Z_i,Z_j)$ with the correlation $\Sigma_{ij}={\rm cor}(Z_i,Z_j)$ for Gaussian distributions:
$$
\Sigma_{ij}\;=\;\sin(\tfrac{\pi}{2}\tau(Z_i,Z_j)).
$$
As observed in \cite{liu2012high}, Kendall's $\tau$ coefficients are the same for	the vector $X$ and for its Gaussian counterpart $Z$ and so we get a simple way of computing the correlation matrix of $Z$ without observing this vector and without knowing the $f_i$'s.

In the special case of tree models, we can extend our distributional
setup arbitrarily. Here however, the test may not be based on querying
the correlation matrix but on potentially more complicated conditional
independence queries. Suppose that a distribution on $\cX\subseteq \R^n$ is given and is $1$-faithful to the graph $G$ in the sense that
\begin{itemize}
	\item [(i)] $X_i\indep X_j$ if and only if $i$ and $j$ lie in two different components of $G$. 
	\item [(ii)] $X_i\indep X_j|X_k$ if and only if $i,j$ lie in two different components of $G\setminus \{k\}$.
\end{itemize}
Assuming that $G$ is connected, our procedure is identical to the one proposed above with the only modification that to decide if $i,j$ lies in the same component of $G\setminus \{k\}$ we use the conditional independence queries. This is also how  we measure the query complexity in this case:
\begin{enumerate}
	\item If $|V|\leq m$, check treeness directly. Otherwise, sample $m$ nodes $W\subseteq V$ as in Section~\ref{sec:central}.
	\item For each $v\in V$ decompose $W$ into $C\cap W$ for $C\in \cC^{(v)}$. To decide if $i,j\in V$ lies in the same block, simply test $X_i\indep X_j|X_v$. 
	\item Take $v^*$ to be the optimal vertex obtained by this procedure. If $M( v^*)>n/2$ then break. Otherwise descend into the components $B=C\cup \{v^*\}$ for $C\in \cC^{(v^*)}$ and go back to Step 1.
\end{enumerate}

Finally, there are non-Gaussian tree models when $\Sigma$ has the same
algebraic structure as in the Gaussian case. This is true for binary
distributions on a tree (i.e., binary Ising models) and more generally for so called linear tree models; see Section~2.3 in \cite{zwiernik2018latent} for definitions. In all these cases, we can apply our tree testing procedure exactly in the same way as presented in the paper.

\section*{Acknowledgements}

We would like to thank Vida Dujmovi\'{c} for helpful remarks.

\appendix

\bibliography{references.bib}
\bibliographystyle{authordate1}

\end{document}